\documentclass[a4paper,11pt,leqno]{smfart}

\usepackage{amssymb,amsmath,enumerate,verbatim,mathrsfs,graphics,graphicx,mathptm,float}
\usepackage[T1]{fontenc}
\usepackage[english, francais]{babel}
\usepackage[all]{xy}
\usepackage[pdfstartview=FitH,
            colorlinks=true,
            linkcolor=blue,
            urlcolor=blue,
            citecolor=red,
            bookmarks,
            bookmarksopen=true,
            bookmarksnumbered=true]{hyperref}

\usepackage{cleveref}
\theoremstyle{plain}

\theoremstyle{definition}

\theoremstyle{remark}

\theoremstyle{plain}
\newtheorem{thmsec}{Th\'eor\`eme}[section]
\newtheorem{thm}[thmsec]{Th\'eor\`eme}
\newtheorem{pro}[thmsec]{Proposition}
\newtheorem{lem}[thmsec]{Lemme}
\newtheorem{cor}[thmsec]{Corollaire}

\theoremstyle{definition}
\newtheorem{defin}[thmsec]{D\'efinition}

\theoremstyle{remark}
\newtheorem{rem}[thmsec]{Remarque}

\newtheorem{eg}[thmsec]{Exemple}

\def\og{\leavevmode\raise.3ex\hbox{$\scriptscriptstyle\langle\!\langle$~}}
\def\fg{\leavevmode\raise.3ex\hbox{~$\!\scriptscriptstyle\,\rangle\!\rangle$}}

\addtocounter{section}{0}             
\numberwithin{equation}{section}       

\newcommand{\N}{\mathbb{N}}
\newcommand{\Z}{\mathbb{Z}}

\newcommand{\C}{\mathbb{C}}
\newcommand{\sph}{\mathbb{P}^{1}_{\mathbb{C}}}
\newcommand{\pp}{\mathbb{P}^{2}_{\mathbb{C}}}
\newcommand{\pd}{\mathbb{\check{P}}^{2}_{\mathbb{C}}}
\newcommand\pgcd{\mathrm{pgcd}}

\newcommand\Sing{\mathrm{Sing}}

\newcommand\Leg{\mathrm{Leg}}
\newcommand\IF{\mathrm{I}_{\mathcal{F}}}
\newcommand\IinvF{\mathrm{I}_{\mathcal{F}}^{\mathrm{inv}}}
\newcommand\ItrF{\mathrm{I}_{\mathcal{F}}^{\hspace{0.2mm}\mathrm{tr}}}
\newcommand\IH{\mathrm{I}_{\mathcal{H}}}
\newcommand\IinvH{\mathrm{I}_{\mathcal{H}}^{\mathrm{inv}}}
\newcommand\ItrH{\mathrm{I}_{\mathcal{H}}^{\hspace{0.2mm}\mathrm{tr}}}
\newcommand\Cinv{\mathrm{C}_{\hspace{-0.3mm}\mathcal{H}}}
\newcommand\Dtr{\mathrm{D}_{\hspace{-0.3mm}\mathcal{H}}}
\newcommand\radH{\Sigma_{\mathcal{H}}^{\mathrm{rad}}}
\newcommand\radHd{\check{\Sigma}_{\mathcal{H}}^{\mathrm{rad}}}
\newcommand\F{\mathcal{F}}
\newcommand\W{\mathcal{W}}
\newcommand\G{\mathcal{G}}
\newcommand\Gunderline{{\mspace{2mu}\underline{\mspace{-2mu}\mathcal{G}\mspace{-2mu}}\mspace{2mu}}}

\begin{document}
\title{Tissus plats et feuilletages homogènes sur le plan projectif}
\date{\today}

\author{Samir \textsc{Bedrouni}}

\address{Facult\'e de Math\'ematiques, USTHB, BP $32$, El-Alia, $16111$ Bab-Ezzouar, Alger, Alg\'erie}
\email{sbedrouni@usthb.dz}

\author{David \textsc{Mar\'{\i}n}}

\address{Departament de Matem\`{a}tiques Universitat Aut\`{o}noma de Barcelona E-08193  Bellaterra (Barcelona) Spain}

\email{davidmp@mat.uab.es}





\maketitle{}

\begin{abstract}
\selectlanguage{french}
Le but de ce travail est d'étudier les feuilletages du plan projectif complexe ayant une transformée de \textsc{Legendre} (tissu dual) plate. Nous établissons quelques critères effectifs de la platitude du $d$-tissu dual d'un feuilletage homogène de degré $d$ et nous  décrivons quelques exemples explicites. Ces résultats nous permettent de montrer qu'à automorphisme de $\pp$ près il y a $11$ feuilletages homogènes de degré $3$ ayant cette propriété. Nous verrons aussi qu'il est possible, sous certaines hypothèses, de ramener l'étude de la platitude du tissu dual d'un feuilletage inhomogène au cadre homogène. Nous obtenons quelques résultats de classification de feuilletages à singularités non-dégénérées et de transformée de  \textsc{Legendre}  plate.
\medskip

\noindent\textit{\textbf{Mots-clés.}} --- tissu, platitude, transformation de \textsc{Legendre}, feuilletage homogène.
\end{abstract}

\begin{altabstract}
\selectlanguage{english}
The aim of this work is to study the foliations on the complex projective plane with flat \textsc{Legendre} transform (dual web). We establish some effective criteria for the flatness of the dual $d$-web of a homogeneous foliation of degree $d$ and we describe some explicit examples. These results allow us to show that up to automorphism of $\pp$ there are $11$ homogeneous foliations of degree $3$ with flat dual web. We will see also that it is possible, under certain assumptions, to bring the study of  flatness of the dual web of a general foliation to the homogeneous framework. We get some classification results about foliations with non-degenerate singularities and flat \textsc{Legendre} transform.
\medskip


\end{altabstract}

\selectlanguage{french}

\section*{Introduction}
\bigskip

\noindent Un $d$-tissu (régulier) $\W$ de $(\mathbb{C}^2,0)$ est la donnée d'une famille $\{\F_1,\F_2,\ldots,\F_d\}$ de feuilletages holomorphes réguliers de $(\mathbb{C}^2,0)$ deux à deux transverses en l'origine. Le premier résultat significatif dans l'étude des tissus a été obtenu par W. \textsc{Blaschke} et J. \textsc{Dubourdieu} autour des années $1920$. Ils ont montré (\cite{BD28}) que tout $3$-tissu régulier $\mathcal{W}$ de $(\mathbb{C}^2,0)$ est conjugué, via un isomorphisme analytique de $(\mathbb{C}^2,0)$, au $3$-tissu trivial défini par $\mathrm{d}x.\mathrm{d}y.\mathrm{d}(x+y)$, et cela sous l'hypothèse d'annulation d'une $2$-forme différentielle $K(\mathcal{W})$ connue sous le nom de courbure de \textsc{Blaschke} de $\mathcal{W}$. La courbure d'un $d$-tissu $\W$ avec~$d>3$ se définit comme la somme des courbures de \textsc{Blaschke} des sous-$3$-tissus de $\W$. Un tissu de courbure nulle est dit plat. Cette notion est utile pour la classification des tissus de rang maximal; un résultat de N. \textsc{Mih\u{a}ileanu} montre que la platitude est une condition nécessaire pour la maximalité du rang, \emph{voir} par exemple \cite{Hen06,Rip07}.

\noindent Depuis peu, l'étude des tissus globaux holomorphes définis sur les surfaces complexes a été réactualisée, \emph{voir} par exemple \cite{PP09,MP13}. Nous nous intéressons dans ce qui suit aux tissus du plan projectif complexe. Un $d$-tissu (global) sur $\pp$ est donné dans une carte affine $(x,y)$ par une équation différentielle algébrique $F(x,y,y')=0$, où $F(x,y,p)=\sum_{i=0}^{d}a_{i}(x,y)p^{d-i}\in\mathbb{C}[x,y,p]$ est un polynôme réduit à coefficient $a_0$ non identiquement nul. Au voisinage de tout point $z_{0}=(x_{0},y_{0})$ tel que $a_{0}(x_{0},y_{0})\Delta(x_{0},y_{0})\neq 0$, où $\Delta(x,y)$ est le $p$-discriminant de $F$, les courbes intégrales de cette équation définissent un $d$-tissu régulier de $(\C^{2},z_{0})$.

\noindent La courbure d'un tissu $\W$ sur $\pp$ est une $2$-forme méromorphe à pôles le long du discriminant $\Delta(\W)$. La platitude d'un tissu $\W$ sur $\pp$ se caractérise par l'holomorphie de sa courbure $K(\W)$ le long des points génériques de $\Delta(\W)$, \emph{voir} \S\ref{subsec:courbure-platitude}.

\noindent D.~\textsc{Mar\'{\i}n} et J. \textsc{Pereira} ont montré, dans \cite{MP13}, comment on peut associer à tout feuilletage $\F$ de degré $d$ sur $\pp$, un $d$-tissu sur le plan projectif dual $\pd$, appelé {\sl transformée de \textsc{Legendre}} de $\F$ et noté $\Leg\F$; les feuilles de $\Leg\F$ sont essentiellement les droites tangentes aux feuilles de $\F$. Plus explicitement, soit $(p,q)$ la carte affine de $\pd$ associée à la droite $\{y=px-q\}\subset{\mathbb{P}^{2}_{\mathbb{C}}}$; si $\F$ est défini par une $1$-forme $\omega=A(x,y)\mathrm{d}x+B(x,y)\mathrm{d}y,$ où $A,B\in\mathbb{C}[x,y],$ $\mathrm{pgcd}(A,B)=1$,  alors $\Leg\F$ est donné par l'équation différentielle algébrique
\[
\check{F}(p,q,x):=A(x,px-q)+pB(x,px-q), \qquad \text{avec} \qquad x=\frac{\mathrm{d}q}{\mathrm{d}p}.
\]
\noindent L'ensemble des feuilletages de degré $d$ sur $\pp$, noté $\mathbf{F}(d)$, s'identifie à un ouvert de \textsc{Zariski} dans un espace projectif de dimension $(d+2)^{2}-2$ sur lequel agit le groupe~$\mathrm{Aut}(\pp).$ Le sous-ensemble $\mathbf{FP}(d)$ de $\mathbf{F}(d)$ formé des $\F\in\mathbf{F}(d)$ tels que $\Leg\F$ soit plat est un fermé de \textsc{Zariski} de $\mathbf{F}(d)$. La classification des feuilletages $\F\in\mathbf{FP}(d)$ modulo $\mathrm{Aut}(\pp)$ reste entière. Le premier cas non trivial que l'on rencontre est celui où $d=3$; on dispose actuellement d'une caractérisation géométrique (\cite[Théorème~4.5]{BFM13}) des éléments de $\mathbf{FP}(3)$, mais ce résultat reste insuffisant pour avancer dans leur classification. C'est dans cette optique que nous nous proposons d'étudier cette question de platitude au niveau des éléments de $\mathbf{F}(d)$ qui sont \textsl{homogènes}, {\it i.e.} qui sont invariants par homothétie. En fait nous établirons, pour des feuilletages homogènes $\mathcal{H}\in\mathbf{F}(d)$, quelques critères effectifs de l'holomorphie de la courbure de $\Leg\mathcal{H}$; de plus nous verrons (Proposition~\ref{pro:F-dégénère-H}) que l'étude de la platitude de la transformée de \textsc{Legendre} d'un feuilletage inhomogène se ramène, sous certaines hypothèses, au cadre homogène.
\smallskip

\noindent Un feuilletage homogène $\mathcal{H}$ de degré $d$ sur $\pp$ est donné, pour un bon choix de coordonnées affines $(x,y)$, par une $1$-forme homogène $\omega_d=A_d(x,y)\mathrm{d}x+B_d(x,y)\mathrm{d}y,$ où $A_d,B_d\in\mathbb{C}[x,y]_d$\, et $\mathrm{pgcd}(A_d,B_d)=1.$

\noindent
L'homogénéité de $\mathcal{H}$ implique (\emph{voir} \cite[page~177]{MP13}) que le discriminant de $\Leg\mathcal{H}$ se décompose en produit de $(d-1)(d+2)$ droites comptées avec multiplicités; certaines parmi elles sont invariantes par $\Leg\mathcal{H}$ et d'autres non, {\it i.e.} sont transverses. De plus la multiplicité de $\Delta(\Leg\mathcal{H})$ le long d'une droite transverse est comprise entre $1$ et $d-1$; en degré $3$ elle est donc soit minimale (égale à $1$) soit maximale (égale à $2$).
\smallskip

\noindent Le Théorème~\ref{thm:holomo-G(I^tr)} affirme que le $d$-tissu $\Leg\mathcal{H}$ est plat si et seulement si sa courbure est holomorphe sur la partie transverse de $\Delta(\Leg\mathcal{H})$.
\smallskip

\noindent Le Théorème~\ref{thm:Barycentre} (resp. Théorème~\ref{thm:Divergence}) contrôle de façon effective l'holomorphie de la courbure $K(\Leg\mathcal{H})$ le long d'une droite $\ell\subset\Delta(\Leg\mathcal{H})$ non invariante par $\Leg\mathcal{H}$ de multiplicité minimale $1$ (resp. maximale $d-1$).

\noindent Ces théorèmes nous permettront de décrire certains feuilletages homogènes appartenant à $\mathbf{FP}(d)$ pour $d$ arbitraire (Propositions
\ref{pro:omega1-omega2}, \ref{pro:omega3-omega4} et \ref{pro:omega5-omega6}).
\smallskip

\noindent En combinant les Théorèmes \ref{thm:holomo-G(I^tr)}, \ref{thm:Barycentre} et \ref{thm:Divergence} nous obtenons une caractérisation complète de la platitude de la transformée de \textsc{Legendre} d'un feuilletage homogène de degré $3$ (Corollaire~\ref{cor:platitude-degre-3}). Ce résultat nous permettra de
classifier les éléments de $\mathbf{FP}(3)$ qui sont homogènes: {\sl à automorphisme de $\pp$ près, il y a $11$ feuilletages homogènes de degré $3$ ayant une transformée de \textsc{Legendre} plate}, \emph{voir} Théorème~\ref{thm:Class-Homog3-Plat}.
\smallskip

\noindent En se basant essentiellement sur cette classification, nous obtenons un résultat (Théorème~\ref{thm:Fermat}) qui sort du cadre homogène: {\sl tout feuilletage $\F\in\mathbf{FP}(3)$ à singularités non-dégénérées ({\it i.e.} ayant pour nombre de \textsc{Milnor}~$1$) est linéairement conjugué au feuilletage de \textsc{Fermat} défini par la $1$-forme $(x^{3}-x)\mathrm{d}y-(y^{3}-y)\mathrm{d}x.$}
\smallskip
Comme application du Théorème~\ref{thm:Fermat} nous donnons une réponse partielle (Corollaire~\ref{cor3}) à \cite[Problème~9.1]{MP13}.

\bigskip
\noindent\textit{\textbf{Remerciements.}}
Ce travail a été soutenu par le Programme National Exceptionnel du Ministère de l'Enseignement Supérieur et de la Recherche Scientifique d'Algérie, et par les projets MTM2011-26674-C02-01 et MTM2015-66165-P du Ministère d'Économie et Compétitivité de l'\'{E}spagne. Le premier auteur remercie le Département de Mathématiques de l'UAB pour son séjour. Il remercie également D. Smaï pour ses précieux conseils.

\section{Préliminaires}
\bigskip

\subsection{Tissus}\label{subsec-tissus}

\noindent Soit $k\geq1$ un entier. Un {\sl $k$-tissu (global)} $\mathcal{W}$ sur une surface complexe $S$ est la donnée d'un recouvrement ouvert $(U_{i})_{i\in I}$ de $S$ et d'une collection de $k$-formes symétriques $\omega_{i}\in \mathrm{Sym}^{k}\Omega^{1}_{S}(U_{i})$, à zéros isolés, satisfaisant:
\begin{itemize}
\item [($\mathfrak{a}$)] il existe $g_{ij}\in \mathcal{O}^{*}_{S}(U_{i}\cap U_{j})$ tel que $\omega_i$ coïncide avec $g_{ij}\omega_{j}$ sur $U_{i}\cap U_{j}$;
\item [($\mathfrak{b}$)] en tout point générique $m$ de $U_{i},$ $\omega_{i}(m)$ se factorise en produit de $k$ formes linéaires deux à deux non colinéaires.
\end{itemize}
L'ensemble des points de $S$ qui ne vérifient pas la condition ($\mathfrak{b}$) est appelé le {\sl discriminant} de $\mathcal{W}$ et est noté $\Delta(\mathcal{W}).$ Lorsque $k=1$ cette condition est toujours vérifiée et on retrouve la définition usuelle d'un feuilletage holomorphe $\mathcal{F}$ sur $S.$ Le cocycle $(g_{ij})$ définit un fibré en droites $N$ sur $S$, appelé le {\sl fibré normal} de $\W$, et les $\omega_i$ se recollent pour définir une section globale $\omega\in\mathrm{H}^0(S,\mathrm{Sym}^{k}\Omega^{1}_{S}\otimes N)$.

\noindent Un $k$-tissu global $\W$ sur $S$ sera dit {\sl décomposable} s'il existe des tissus globaux $\mathcal{W}_{1},\mathcal{W}_{2}$ sur $S$ n'ayant pas de sous-tissus communs tels que $\mathcal{W}$ soit la superposition de $\mathcal{W}_{1}$ et $\mathcal{W}_{2}$; on écrira $\mathcal{W}=\mathcal{W}_{1}\boxtimes\mathcal{W}_{2}.$ Dans le cas contraire $\W$ sera dit {\sl irréductible}. On dira que $\mathcal{W}$ est {\sl complètement décomposable} s'il existe des feuilletages globaux $\mathcal{F}_{1},\ldots,\mathcal{F}_{k}$ sur $S$ tels que $\mathcal{W}=\mathcal{F}_{1}\boxtimes\cdots\boxtimes\mathcal{F}_{k}.$ Pour en savoir plus à ce sujet, nous renvoyons à \cite{PP09}.
\smallskip

\noindent On se restreindra dans ce travail au cas $S=\pp$. Se donner un $k$-tissu sur $\pp$ revient à se donner une $k$-forme symétrique polynomiale $\omega=\sum_{i+j=k}a_{ij}(x,y)\mathrm{d}x^{i}\mathrm{d}y^{j}$, à zéros isolés et de discriminant non identiquement nul. Ainsi tout $k$-tissu sur $\pp$ peut se lire dans une carte affine donnée $(x,y)$ de $\pp$ par une équation différentielle polynomiale $F(x,y,y')=0$ de degré $k$ en $y'$. Un $k$-tissu $\W$ sur $\pp$ est dit de {\sl degré} $d$ si le nombre de points où une droite générique de $\pp$ est tangente à une feuille de $\W$ est égal à $d$; c'est équivalent de dire que $\W$ est de fibré normal $N=\mathcal{O}_{\pp}(d+2k).$ Il est bien connu, {\em voir} par exemple \cite[Proposition~1.4.2]{PP09}, que les tissus de degré $0$ sont les tissus algébriques (leurs feuilles sont les droites tangentes à une courbe algébrique réduite).

\noindent Les auteurs dans \cite{MP13} ont associé, à tout $k$-tissu $\W$ de degré $d\geq1$ sur $\pp$, un $d$-tissu de degré $k$ sur le plan projectif dual $\pd$, appelé {\sl transformée de \textsc{Legendre}} de $\W$ et noté $\Leg\W$; les feuilles de $\Leg\W$ sont essentiellement les droites tangentes aux feuilles de $\W$. Plus explicitement, soit $(x,y)$ une carte affine de $\pp$ et considérons la carte affine  $(p,q)$  de $\pd$ associée à la droite $\{y=px-q\}\subset{\mathbb{P}^{2}_{\mathbb{C}}}.$ Soit $F(x,y;p)=0$, $p=\frac{\mathrm{d}y}{\mathrm{d}x}$, une équation différentielle implicite décrivant $\W$; alors $\Leg\W$ est donné par l'équation différentielle implicite
\[
\check{F}(p,q;x):=F(x,px-q;p)=0, \qquad \text{avec} \qquad x=\frac{\mathrm{d}q}{\mathrm{d}p}.
\]

\noindent Il est clair que cette transformation est involutive, {\it i.e.} $\Leg(\Leg\W)=\W$. Notons enfin que si $\mathcal{F}$ est un feuilletage de degré $d\geq1$ sur $\pp$, alors $\Leg\F$ est un $d$-tissu irréductible de degré $1$ sur $\pd$. Inversement un $d$-tissu irréductible de degré $1$ sur $\pd$ est nécessairement la transformée de \textsc{Legendre} d'un certain feuilletage de degré $d$ sur $\pp$ (\emph{voir} \cite{MP13}).

\subsection{Courbure et platitude}\label{subsec:courbure-platitude}

On rappelle ici la définition de la courbure d'un $k$-tissu $\mathcal{W}.$ On suppose dans un premier temps que $\mathcal{W}$ est un germe de $k$-tissu de $(\mathbb{C}^{2},0)$ complètement décomposable, $\mathcal{W}=\mathcal{F}_{1}\boxtimes\cdots\boxtimes\mathcal{F}_{k}.$ Soit, pour tout $1\leq i\leq k,$ une $1$-forme $\omega_{i}$ à singularité isolée en $0$ définissant le feuilletage $\mathcal{F}_{i}.$ D'après \cite{PP08}, pour tout triplet $(r,s,t)$ avec $1\leq r<s<t\leq k,$ on définit $\eta_{rst}=\eta(\mathcal{F}_{r}\boxtimes\mathcal{F}_{s}\boxtimes\mathcal{F}_{t})$ comme l'unique $1$-forme méromorphe satisfaisant les égalités suivantes:
\begin{equation}\label{equa:eta-rst}
{\left\{\begin{array}[c]{lll}
\mathrm{d}(\delta_{st}\,\omega_{r}) &=& \eta_{rst}\wedge\delta_{st}\,\omega_{r}\\
\mathrm{d}(\delta_{tr}\,\omega_{s}) &=& \eta_{rst}\wedge\delta_{tr}\,\omega_{s}\\
\mathrm{d}(\delta_{rs}\,\omega_{t}) &=& \eta_{rst}\wedge\delta_{rs}\,\omega_{t}
\end{array}
\right.}
\end{equation}
où $\delta_{ij}$ désigne la fonction définie par $\omega_{i}\wedge\omega_{j}=\delta_{ij}\,\mathrm{d}x\wedge\mathrm{d}y.$ Comme chacune des $1$-formes $\omega_{i}$ n'est définie qu'à multiplication près par un inversible de $\mathcal{O}(\mathbb{C}^{2},0),$ il en résulte que chacune des $1$-formes $\eta_{rst}$ est bien déterminée à l'addition près d'une $1$-forme holomorphe fermée. Ainsi la $1$-forme
\begin{equation}\label{equa:eta}
\hspace{7mm}\eta(\mathcal{W})=\eta(\mathcal{F}_{1}\boxtimes\cdots\boxtimes\mathcal{F}_{k})=\sum_{1\le r<s<t\le k}\eta_{rst}
\end{equation}
est bien définie à l'addition près d'une $1$-forme holomorphe fermée. La {\sl courbure} du tissu $\mathcal{W}=\mathcal{F}_{1}\boxtimes\cdots\boxtimes\mathcal{F}_{k}$ est par définition la $2$-forme
\begin{align*}
&K(\mathcal{W})=K(\mathcal{F}_{1}\boxtimes\cdots\boxtimes\mathcal{F}_{k})=\mathrm{d}\,\eta(\mathcal{W}).
\end{align*}
\noindent On peut vérifier que $K(\mathcal{W})$ est une $2$-forme méromorphe à pôles le long du discriminant $\Delta(\mathcal{W})$ de $\mathcal{W},$ canoniquement associée à $\mathcal{W}$; plus précisément, pour toute application holomorphe dominante $\varphi,$ on a $K(\varphi^{*}\mathcal{W})=\varphi^{*}K(\mathcal{W}).$
\smallskip

\noindent Si maintenant $\mathcal{W}$ est un $k$-tissu sur une surface complexe $S$ (non forcément complètement décomposable), alors on peut le transformer en un $k$-tissu complètement décomposable au moyen d'un revêtement galoisien ramifié. L'invariance de la courbure de ce nouveau tissu par l'action du groupe de \textsc{Galois} permet de la redescendre en une $2$-forme méromorphe globale sur $S,$ à pôles le long du discriminant de $\mathcal{W}$ (\emph{voir} \cite{MP13}).
\smallskip

\noindent Un $k$-tissu $\mathcal{W}$ est dit {\sl plat} si sa courbure $K(\mathcal{W})$ est identiquement nulle.
\smallskip

\noindent Signalons qu'un $k$-tissu $\mathcal{W}$ sur $\mathbb{P}^{2}_{\mathbb{C}}$ est plat si et seulement si sa courbure est holomorphe le long des points génériques des composantes irréductibles de $\Delta(\mathcal{W})$. Ceci résulte de la définition de $K(\mathcal{W})$ et du fait qu'il n'existe pas de $2$-forme holomorphe sur $\mathbb{P}^{2}_{\mathbb{C}}$ autre que la $2$-forme nulle.

\subsection{Singularités et diviseur d'inflexion d'un feuilletage du plan projectif}

Un feuilletage holomorphe $\mathcal{F}$ de degré $d$ sur~$\pp$ est défini par une $1$-forme du type $$\omega=a(x,y,z)\mathrm{d}x+b(x,y,z)\mathrm{d}y+c(x,y,z)\mathrm{d}z,$$ o\`{u} $a,$ $b$ et $c$ sont des polynômes homogènes de degré $d+1$ sans composante commune satisfaisant la condition d'\textsc{Euler} $i_{\mathrm{R}}\omega=0$, où $\mathrm{R}=x\frac{\partial{}}{\partial{x}}+y\frac{\partial{}}{\partial{y}}+z\frac{\partial{}}{\partial{z}}$ désigne le champ radial et $i_{\mathrm{R}}$ le produit intérieur par $\mathrm{R}$. Le {\sl lieu singulier} $\mathrm{Sing}\mathcal{F}$ de $\mathcal{F}$ est le projectivisé du lieu singulier de~$\omega$ $$\mathrm{Sing}\omega=\{(x,y,z)\in\mathbb{C}^3\,\vert \, a(x,y,z)=b(x,y,z)=(x,y,z)=0\}.$$

\noindent Rappelons quelques notions locales attachées au couple $(\mathcal{F},s)$, où $s\in\Sing\mathcal{F}$. Le germe de $\F$ en $s$ est défini, à multiplication près par une unité de l'anneau local $\mathcal{O}_s$ en $s$, par un champ de vecteurs
\begin{small}
$\mathrm{X}=A(\mathrm{u},\mathrm{v})\frac{\partial{}}{\partial{\mathrm{u}}}+B(\mathrm{u},\mathrm{v})\frac{\partial{}}{\partial{\mathrm{v}}}$.
\end{small}
\noindent La {\sl multiplicité algébrique} $\nu(\mathcal{F},s)$ de $\mathcal{F}$ en $s$ est donnée par $$\nu(\mathcal{F},s)=\min\{\nu(A,s),\nu(B,s)\},$$ où $\nu(g,s)$ désigne la multiplicité algébrique de la fonction $g$ en $s$. L'{\sl ordre de tangence} entre $\mathcal{F}$ et une droite générique passant par $s$ est l'entier $$\hspace{0.8cm}\tau(\mathcal{F},s)=\min\{k\geq\nu(\mathcal{F},s)\hspace{1mm}\colon\det(J^{k}_{s}\,\mathrm{X},\mathrm{R}_{s})\neq0\},$$ où $J^{k}_{s}\,\mathrm{X}$ est le $k$-jet de $\mathrm{X}$ en $s$ et $\mathrm{R}_{s}$ est le champ radial centré en $s$. Le \textsl{nombre de \textsc{Milnor}} de $\F$ en $s$ est l'entier $$\mu(\mathcal{F},s)=\dim_\mathbb{C}\mathcal{O}_s/\langle A,B\rangle,$$ où $\langle A,B\rangle$ désigne l'id\'eal de $\mathcal{O}_s$ engendr\'e par $A$ et $B$.
\smallskip

\noindent La singularité $s$ est dite {\sl radiale d'ordre} $n-1$ si $\nu(\mathcal{F},s)=1$ et $\tau(\mathcal{F},s)=n.$

\smallskip
\noindent La singularité $s$ est dite {\sl non-dégénérée} si $\mu(\F,s)=1$, c'est équivalent de dire que la partie linéaire $J^{1}_{s}\mathrm{X}$ de $\mathrm{X}$ possède deux valeurs propres $\lambda,\mu$ non nulles. La quantité $\mathrm{BB}(\F,s)=\frac{\lambda}{\mu}+\frac{\mu}{\lambda}+2$ est appelée l'{\sl invariant de \textsc{Baum-Bott}} de $\F$ en $s$ (\emph{voir} \cite{BB72}). D'après \cite{CS82} il passe par $s$ au moins un germe de courbe $\mathcal{C}$ invariante par $\F$; à isomorphisme local près, on peut se ramener à $s=(0,0)$,\, $\mathrm{T}_{s}\mathcal{C}=\{\mathrm{v}=\,0\,\}$ et $J^{1}_{s}\mathrm{X}=\lambda \mathrm{u}\frac{\partial}{\partial\mathrm{u}}+(\varepsilon \mathrm{u}+\mu\hspace{0.1mm}\mathrm{v})\frac{\partial}{\partial \mathrm{v}}$, où l'on peut prendre $\varepsilon=0$ si $\lambda\neq\mu$. La quantité $\mathrm{CS}(\F,\mathcal{C},s)=\frac{\lambda}{\mu}$ est appelée l'{\sl indice de \textsc{Camacho-Sad}} de $\F$ en $s$ par rapport à $\mathcal{C}$.
\medskip

\noindent Rappelons la notion du diviseur d'inflexion de $\F$. Soit $\mathrm{Z}=E\frac{\partial}{\partial x}+F\frac{\partial}{\partial y}+G\frac{\partial}{\partial z}$ un champ de vecteurs homogène de degré $d$ sur $\mathbb{C}^3$ non colinéaire au champ radial décrivant $\mathcal{F},$ {\it i.e.} tel que $\omega=i_{\mathrm{R}}i_{\mathrm{Z}}\mathrm{d}x\wedge\mathrm{d}y\wedge\mathrm{d}z.$ Le {\sl diviseur d'inflexion} de $\mathcal{F}$, noté $\IF$, est le diviseur défini par l'équation
\begin{equation}\label{equa:ext1}
\left| \begin{array}{ccc}
x &  E &  \mathrm{Z}(E) \\
y &  F &  \mathrm{Z}(F)  \\
z &  G &  \mathrm{Z}(G)
\end{array} \right|=0.
\end{equation}
Ce diviseur a été étudié dans \cite{Per01} dans un contexte plus général. En particulier, les propriétés suivantes ont été prouvées.
\begin{enumerate}
\item [\texttt{1.}] Sur $\mathbb{P}^{2}_{\mathbb{C}}\smallsetminus\mathrm{Sing}\mathcal{F},$ $\IF$ coïncide avec la courbe décrite par les points d'inflexion des feuilles de $\mathcal{F}$;
\item [\texttt{2.}] Si $\mathcal{C}$ est une courbe algébrique irréductible invariante par $\mathcal{F},$ alors $\mathcal{C}\subset \IF$ si et seulement si $\mathcal{C}$ est une droite invariante;
\item [\texttt{3.}] $\IF$ peut se décomposer en $\IF=\IinvF+\ItrF,$ où le support de $\IinvF$ est constitué de l'ensemble des droites invariantes par $\mathcal{F}$ et où le support de $\ItrF$ est l'adhérence des points d'inflexion qui sont isolés le long des feuilles de $\mathcal{F}$;
\item [\texttt{4.}] Le degré du diviseur $\IF$ est $3d.$
\end{enumerate}

\noindent Le feuilletage $\mathcal{F}$ sera dit {\sl convexe} si son diviseur d'inflexion $\IF$ est totalement invariant par $\mathcal{F}$, {\it i.e.} si $\IF$ est le produit de droites invariantes.

\noindent L'application de Gauss est l'application rationnelle $\mathcal{G}_{\F}\hspace{1mm}\colon\pp\dashrightarrow \pd$ qui à un point régulier $m$ associe la droite tangente $\mathrm{T}_{\hspace{-0.1mm}m}\F$. Si $\mathcal{C}\subset\pp$ est une courbe passant par certains points singuliers de $\F$, on définit $\mathcal{G}_{\mathcal{F}}(\mathcal{C})$ comme étant l'adhérence de $\mathcal{G}_{\F}(\mathcal{C}\setminus\Sing\F)$. Il résulte de \cite[Lemme~2.2]{BFM13} que
\begin{equation}\label{equa:Delta-LegF}
\Delta(\Leg\F)=\mathcal{G}_{\F}(\ItrF)\cup\check{\Sigma}_{\F},
\end{equation}
où $\check{\Sigma}_{\F}$ désigne l'ensemble des droites duales des points de $\Sigma_{\F}:=\{s\in\Sing\F\hspace{0.8mm}:\hspace{0.8mm}\tau(\F,s)\geq2\}$.

\section{Géométrie des feuilletages homogènes}
\bigskip

\begin{defin}\label{def:feuilletage-homog}
Un feuilletage de degré $d$ sur $\mathbb{P}^{2}_{\mathbb{C}}$ est dit {\sl homogène} s'il existe une carte affine $(x,y)$ de $\mathbb{P}^{2}_{\mathbb{C}}$ dans laquelle il est invariant sous l'action du groupe des homothéties $(x,y)\longmapsto \lambda(x,y),\hspace{1mm} \lambda\in \mathbb{C}^{*}.$
\end{defin}

\noindent Un tel feuilletage $\mathcal{H}$ est alors défini par une $1$-forme $$\omega=A(x,y)\mathrm{d}x+B(x,y)\mathrm{d}y,$$ où $A$ et $B$ sont des polynômes homogènes de degré $d$ sans composante commune. Cette $1$-forme s'écrit en coordonnées homogènes
\begin{align*}
&\hspace{2mm} z\hspace{0.3mm}A(x,y)\mathrm{d}x+z\hspace{0.2mm}B(x,y)\mathrm{d}y-\left(x\hspace{0.2mm}A(x,y)+yB(x,y)\right)\mathrm{d}z\hspace{1mm};
\end{align*}
ainsi le feuilletage $\mathcal{H}$ a au plus $d+2$ singularités dont l'origine $O$ de la carte affine $z=1$ est le seul point singulier de $\mathcal{H}$ qui n'est pas situé sur la droite à l'infini $L_{\infty}=(z=0)$; de plus $\nu(\mathcal{H},O)=d.$

\noindent Dorénavant nous supposerons que $d$ est supérieur ou égal à $2.$ Dans ce cas le point $O$ est la seule singularité de $\mathcal{H}$ de multiplicité algébrique  $d.$

\noindent Le champ de vecteurs homogène $-B(x,y)\frac{\partial}{\partial x}+A(x,y)\frac{\partial}{\partial y}+0\frac{\partial}{\partial z}$ défini aussi le feuilletage $\mathcal{H}$ car est dans le noyau de la $1$-forme précédente; d'après la formule (\ref{equa:ext1}), le diviseur d'inflexion $\IH$ de $\mathcal{H}$ est donné par
\begin{equation*}
0=\left| \begin{array}{ccc}
x & \hspace{-1.5mm}-B &  BB_{x}-AB_{y} \\
y &  A &  AA_{y}-BA_{x}  \\
z &  0 &  0
\end{array} \right|
=z
\left| \begin{array}{cc}
-\frac{1}{d}(xB_{x}+yB_{y})                &  BB_{x}-AB_{y} \vspace{1.5mm}\\
\hspace{2.4mm}\frac{1}{d}(xA_{x}+yA_{y})   &  AA_{y}-BA_{x}
\end{array} \right|
=\frac{z}{d}(xA+yB)(A_{x}B_{y}-A_{y}B_{x})
=\frac{z}{d}\Cinv\Dtr,
\end{equation*}
où $\Cinv=xA+yB\in\mathbb{C}[x,y]_{d+1}$ désigne le {\sl cône tangent} de $\mathcal{H}$ en l'origine $O$ et $\Dtr=A_{x}B_{y}-A_{y}B_{x}\in\mathbb{C}[x,y]_{2d-2}$.

\noindent Il en résulte que:
\begin{enumerate}
\item [(i)] le support du diviseur $\IinvH$ est constitué des droites du cône tangent $\Cinv=0$ et de la droite à l'infini $L_{\infty}$;
\item [(ii)] le diviseur $\ItrH$ se décompose sous la forme $\ItrH=\prod_{i=1}^{n}T_{i}^{\rho_{i}-1}$ pour un certain nombre $n\leq \deg\Dtr=2d-2$ de droites $T_{i}$ passant par $O,$ $\rho_{i}-1$ étant l'ordre d'inflexion de la droite $T_{i}.$ Lorsque $\rho_{i}=2$ on parle d'une droite d'inflexion simple pour $\mathcal{H},$ lorsque $\rho_{i}=3$ d'une droite d'inflexion double, etc.
\end{enumerate}

\begin{pro}\label{pro:SingH}
{\sl Avec les notations précédentes, pour tout point singulier $s\in\mathrm{Sing}\mathcal{H}\cap L_{\infty},$ nous avons

\noindent\textbf{\textit{1.}} $\nu(\mathcal{H},s)=1;$

\noindent\textbf{\textit{2.}} la droite $L_{s}$ passant par l'origine $O$ et le point $s$ est invariante par $\mathcal{H}$ et elle apparaît avec multiplicité $\tau(\mathcal{H},s)-1$ dans le diviseur $\Dtr=0,$ {\it i.e.}
$$\Dtr=\ItrH\prod_{s\in\mathrm{Sing}\mathcal{H}\cap L_{\infty}}L_{s}^{\tau(\mathcal{H},s)-1}.$$
}
\end{pro}

\begin{proof}[\sl D\'emonstration]
Soit $s$ un point singulier de $\mathcal{H}$ sur $L_{\infty}=(z=0)$. Sans perte de généralité, nous pouvons supposer que les coordonnées homogènes de $s$ sont de la forme $[x_0:1:0],\,x_0\in\mathbb{C}.$ Dans la carte affine $y=1$, $\mathcal{H}$ est décrit par la $1$-forme
\begin{align*}
&\theta=z\hspace{0.3mm}A(x,1)\mathrm{d}x-\left(x\hspace{0.2mm}A(x,1)+B(x,1)\right)\mathrm{d}z\hspace{1mm};
\end{align*}
la condition $s\in\Sing\mathcal{H}$ est équivalente à $B(x_0,1)=-x_0\hspace{0.4mm}A(x_0,1).$ L'égalité $\pgcd(A,B)=1$ implique alors que $A(x_{0},1)\neq0$; d'où $\nu(\mathcal{H},s)=1.$

\noindent Montrons la seconde assertion. Le fait que $$\theta=A(x,1)\left(z\mathrm{d}(x-x_0)-(x-x_0)\mathrm{d}z\right)-\left(x_0\hspace{0.4mm}A(x,1)+B(x,1)\right)\mathrm{d}z$$
entraîne que $$\tau:=\tau(\mathcal{H},s)=\min\{k\geq1:J^{k}_{x_0}(x_0\hspace{0.4mm}A(x,1)+B(x,1))\neq0\},$$
cela permet d'écrire $x_0\hspace{0.4mm}A(x,1)+B(x,1)=\sum_{k=\tau}^{d}c_{k}(x-x_0)^{k}$, avec $c_{\tau}\neq0$. Par suite
\begin{align*}
&& B(x,y)=(x-x_0y)^{\tau}P(x,y)-x_0\hspace{0.4mm}A(x,y), \quad \text{où} \quad P(x,y)=\sum_{k=0}^{d-\tau}c_{k+\tau}(x-x_0y)^{k}y^{d-\tau-k}.
\end{align*}
Un calcul élémentaire montre que $\Dtr=A_{x}B_{y}-A_{y}B_{x}$ est de la forme $\Dtr=-(x-x_0y)^{\tau-1}Q(x,y),$ avec $Q\in\mathbb{C}[x,y]$ et $$Q(x_0,1)=\tau P(x_0,1)\left(xA_x+yA_y\right)\Big|_{(x,y)=(x_0,1)}.$$
Comme $P(x_0,1)=c_{\tau}$\, et \,$xA_x+yA_y=dA$,\, $Q(x_0,1)=\tau\hspace{0.2mm}c_{\tau}\hspace{0.2mm}dA(x_0,1)\neq0.$
\end{proof}

\begin{defin}\label{def:type-homog}
Soit $\mathcal{H}$ un feuilletage homogène de degré $d$ sur $\pp$ ayant un certain nombre $m\leq2d-2$ de singularités radiales $s_{i}$ d'ordre $\tau_{i}-1,$ $2\leq\tau_{i}\leq d$ pour $i=1,2,\ldots,m.$ Le support du diviseur $\ItrH$ est constitué d'un certain nombre $n\leq2d-2$ de droites d'inflexion transverse $T_{j}$ d'ordre $\rho_{j}-1,$ $2\leq\rho_{j}\leq d$ pour $j=1,2,\ldots,n.$ On définit le {\sl type du feuilletage} $\mathcal{H}$ par $$\mathcal{T}_\mathcal{H}=\sum\limits_{i=1}^{m}\mathrm{R}_{\tau_{i}-1}+\sum\limits_{j=1}^{n}\mathrm{T}_{\rho_{j}-1}=
\sum\limits_{k=1}^{d-1}(r_{k}\cdot\mathrm{R}_k+t_{k}\cdot\mathrm{T}_k)
\in\Z\left[\mathrm{R}_1,\mathrm{R}_2,\ldots,\mathrm{R}_{d-1},\mathrm{T}_1,\mathrm{T}_2,\ldots,\mathrm{T}_{d-1}\right]$$ et le {\sl degré du type} $\mathcal{T}_\mathcal{H}$ par $\deg\mathcal{T}_\mathcal{H}=\sum_{k=1}^{d-1}(r_{k}+t_{k})\in\N\setminus\{0,1\}$; c'est le nombre de droites distinctes qui composent le diviseur $\Dtr.$
\end{defin}

\begin{eg}
Considérons le feuilletage homogène $\mathcal{H}$ de degré $5$ sur $\pp$ défini par $$\omega=y^5\mathrm{d}x+2x^3(3x^2-5y^2)\mathrm{d}y.$$ Un calcul élémentaire conduit à
\begin{align*}
&& \mathrm{C}_{\mathcal{H}}=xy\left(6x^4-10x^2y^2+y^4\right) \qquad\text{et}\qquad \Dtr=150\hspace{0.15mm}x^2y^4(x-y)(x+y)\hspace{1mm};
\end{align*}
on constate que l'ensemble des singularités radiales de $\mathcal{H}$ est constitué des deux points $[0:1:0]$ et $[1:0:0]$; leurs ordres de radialité sont égaux respectivement à $2$ et $4.$ De plus le support du diviseur $\ItrH$ est formé des deux droites d'équations $x-y=0$ et $x+y=0$; ce sont des droites d'inflexion transverse simple. Donc le feuilletage $\mathcal{H}$ est du type $\mathcal{T}_\mathcal{H}=1\cdot\mathrm{R}_{2}+1\cdot\mathrm{R}_{4}+2\cdot\mathrm{T}_{1}$ et le degré de $\mathcal{T}_\mathcal{H}$ est $\deg\mathcal{T}_\mathcal{H}=4.$
\end{eg}

\noindent \`{A} tout feuilletage homogène $\mathcal{H}$ de degré $d$ sur $\pp$ on peut associer une application rationnelle $\Gunderline_{\mathcal{H}}\hspace{1mm}\colon\mathbb{P}^{1}_{\mathbb{C}}\rightarrow \mathbb{P}^{1}_{\mathbb{C}}$ de la façon suivante: si $\mathcal{H}$ est décrit par $\omega=A(x,y)\mathrm{d}x+B(x,y)\mathrm{d}y,$ $A$ et $B$ désignant des polynômes homogènes de degré $d$ sans facteur commun, on définit $\Gunderline_{\mathcal{H}}$ par $$\Gunderline_{\mathcal{H}}([x:y])=[-A(x,y):B(x,y)]\hspace{1mm};$$ il est clair que cette définition ne dépend pas du choix de la $1$-forme homogène $\omega$ décrivant le feuilletage $\mathcal{H}.$

\noindent Dorénavant nous noterons l'application $\Gunderline_{\mathcal{H}}$ simplement par $\Gunderline.$ Le feuilletage homogène $\mathcal{H}$ ainsi que son tissu dual $\mathrm{Leg}\mathcal{H}$ peuvent être décrits analytiquement en utilisant uniquement l'application $\Gunderline.$ En effet, la pente $p$ de $\mathrm{T}_{\hspace{-0.4mm}(x,y)}\mathcal{H}$ est donnée par $\Gunderline([x:y])=[p:1]$ et les pentes $x_{i}$ $(i=1,\ldots,d)$ de $\mathrm{T}_{\hspace{-0.4mm}(p,q)}\mathrm{Leg}\mathcal{H}$ sont données par $x_{i}=\dfrac{q}{p-p_{i}(p)},$ avec $\Gunderline^{-1}([p:1])=\{[p_{i}(p):1]\}.$

\noindent En carte affine $\mathbb{C}\subset\sph$ cette application s'écrit $\Gunderline\hspace{1mm}\colon z\mapsto-\dfrac{A(1,z)}{B(1,z)}.$ On a
\begin{align*}
\Gunderline(z)-z=-\dfrac{A(1,z)+zB(1,z)}{B(1,z)}=-\frac{\mathrm{C}_\mathcal{H}(1,z)}{B(1,z)}\hspace{1mm};
\end{align*}
\noindent de plus, les identités $dA=xA_x+yA_y$\hspace{1mm} et \hspace{1mm}$dB=xB_x+yB_y$
permettent de réécrire $\Dtr$ sous la forme  $\Dtr=-\dfrac{\raisebox{-0.5mm}{$d$}}{\raisebox{0.7mm}{$x$}}\left(BA_y-AB_y\right)$
de sorte que
\begin{align*}
\hspace{5mm}\Gunderline'(z)=-\left(\frac{BA_y-AB_y}{B^{2}}\right)\Big|_{(x,y)=(1,z)}=\frac{\Dtr(1,z)}{dB^{2}(1,z)}.
\end{align*}
\noindent On en déduit immédiatement les propriétés suivantes:

\begin{itemize}
\item [\texttt{1.}] les points fixes de $\Gunderline$ correspondent au cône tangent de $\mathcal{H}$ en l'origine $O$ (\textit{i.e.}  $[a:b]\in\mathbb{P}^{1}_{\mathbb{C}}$ est fixe par $\Gunderline$ si et seulement si la droite d'équation $by-a\hspace{0.2mm}x=0$ est invariante par $\mathcal{H}$\hspace{0.25mm});

\item [\texttt{2.}] le point $[a:b]\in\mathbb{P}^{1}_{\mathbb{C}}$ est critique fixe par $\Gunderline$ si et seulement si le point $[b:a:0]\in L_{\infty}$ est singulier radial de $\mathcal{H}$. La multiplicité du point critique $[a:b]$ de $\Gunderline$ est exactement égale à l'ordre de radialité de la singularité à l'infini;

\item [\texttt{3.}] le point $[a:b]\in\mathbb{P}^{1}_{\mathbb{C}}$ est critique non fixe par $\Gunderline$ si et seulement si la droite d'équation $by-a\hspace{0.2mm}x=0$ est une droite d'inflexion transverse pour $\mathcal{H}.$ La multiplicité du point critique $[a:b]$ de $\Gunderline$ est précisément égale à l'ordre d'inflexion de cette droite.
\end{itemize}

\begin{rem}\label{rem:2d-2-R1}
Pour qu'un feuilletage homogène de degré $d\geq2$ sur $\pp$ soit convexe de type $(2d-2)\cdot \mathrm{R}_1$ il faut que $d\in\{2,3\},$ car tout feuilletage homogène de degré $d$ sur $\pp$ a au plus $d+1$ points singuliers à l'infini. En fait, même en degré $d\in\{2,3\}$, le type $(2d-2)\cdot \mathrm{R}_1$ ne se produit pas. Ceci découle du fait bien connu qu'une application rationnelle de la sphère de \textsc{Riemann} dans elle-même a au moins un point fixe non critique (\emph{voir} par exemple \cite[Théorème~12.4]{Mil99}).
\end{rem}

\section{\'{E}tude de la platitude du tissu dual d'un feuilletage homogène}\label{sec:étude-platitude-homog}
\bigskip

\noindent La Proposition~$3.2$ de \cite{BFM13} est un critère de la platitude de la transformée de \textsc{Legendre} d'un feuilletage homogène de degré $3$. Notre premier résultat généralise ce critère en degré arbitraire.

\begin{thm}\label{thm:holomo-G(I^tr)}
{\sl Soit $\mathcal{H}$ un feuilletage homogène de degré $d\geq3$ sur $\pp.$ Alors le $d$-tissu $\Leg\mathcal{H}$ est plat si et seulement si sa courbure $K(\Leg\mathcal{H})$ est holomorphe sur $\G_{\mathcal{H}}(\ItrH).$
}
\end{thm}

\noindent Dans tout ce qui suit, $\mathcal{H}$ désigne un feuilletage homogène de degré $d\geq3$ sur $\pp$ défini, en carte affine $(x,y)$, par la $1$-forme
\begin{align*}
& \omega=A(x,y)\mathrm{d}x+B(x,y)\mathrm{d}y,\quad A,B\in\mathbb{C}[x,y]_{d},\hspace{2mm}\pgcd(A,B)=1.
\end{align*}

\noindent La démonstration de ce théorème utilise les deux lemmes suivants.

\begin{lem}\label{lem:Delta-LegH}
{\sl Le discriminant de $\Leg\mathcal{H}$ se décompose en $$\Delta(\Leg\mathcal{H})=\G_{\mathcal{H}}(\ItrH)\cup\radHd\cup\check{O},$$ où $\radHd$ désigne l'ensemble des droites duales des points de $\radH=\{s\in\Sing\mathcal{H}\hspace{0.8mm}:\hspace{0.8mm}\nu(\mathcal{H},s)=1,\,\tau(\mathcal{H},s)\geq2\}.$
}
\end{lem}

\begin{proof}[\sl D\'emonstration]
La formule (\ref{equa:Delta-LegF}) nous donne $\Delta(\Leg\mathcal{H})=\mathcal{G}_{\mathcal{H}}(\ItrH)\cup\check{\Sigma}_{\mathcal{H}}$, où $\check{\Sigma}_{\mathcal{H}}$ est l'ensemble des droites duales des points de $\Sigma_{\mathcal{H}}=\{s\in\Sing\mathcal{H}\hspace{0.8mm}:\hspace{0.8mm}\tau(\mathcal{H},s)\geq2\}$. D'après la première assertion de la Proposition~\ref{pro:SingH}, l'origine $O$ est le seul point singulier de $\mathcal{H}$ de multiplicité algébrique supérieure ou égale à $2$; par conséquent $\Sigma_{\mathcal{H}}=\radH\cup\{O\}.$
\end{proof}

\begin{lem}\label{lem:holomo-O}
{\sl
Si la courbure de $\Leg\mathcal{H}$ est holomorphe sur $\pd\hspace{-0.3mm}\setminus\hspace{-0.3mm} \check{O},$ alors $\Leg\mathcal{H}$ est plat.
}
\end{lem}

\begin{proof}[\sl D\'emonstration]
Soit $(a,b)$ la carte affine de $\pd$ associée à la droite $\{ax-by+1=0\}\subset{\pp}$; le $d$-tissu $\mathrm{Leg}\mathcal{H}$ est donné par la $d$-forme symétrique $\check{\omega}=bA(\mathrm{d}b,\mathrm{d}a)+aB(\mathrm{d}b,\mathrm{d}a)$. L'homogénéité de $A$ et $B$ implique alors que toute homothétie $h_{\lambda}\hspace{1mm}\colon(a,b)\longmapsto\lambda(a,b)$ laisse invariant $\Leg \mathcal{H}$; par suite $$h_{\lambda}^{*}(K(\Leg \mathcal{H}))=K(\Leg \mathcal{H}).$$
En combinant l'hypothèse de l'holomorphie de la courbure en dehors de $\check{O}$ avec le fait que $\check{O}$ est la droite à l'infini dans la carte $(a,b),$ on constate que $K(\Leg \mathcal{H})=P(a,b)\mathrm{d}a\wedge\mathrm{d}b$ pour un certain $P\in\mathbb{C}[a,b].$ On déduit de ce qui précède que $\lambda^{2}P(\lambda\hspace{0.1mm}a,\lambda\hspace{0.1mm}b)=P(a,b),$ d'où l'énoncé.
\end{proof}

\begin{proof}[\sl D\'emonstration du Théorème~\ref{thm:holomo-G(I^tr)}]
L'implication directe est triviale. Montrons la réciproque; supposons que $K(\Leg\mathcal{H})$ soit holomorphe sur $\G_{\mathcal{H}}(\ItrH).$ D'après les Lemmes \ref{lem:Delta-LegH} et \ref{lem:holomo-O}, il suffit de prouver que $K(\Leg\mathcal{H})$ est holomorphe le long de $\Xi:=\radHd\setminus\G_{\mathcal{H}}(\ItrH).$ Supposons donc $\Xi$ non vide; soit $s$ une singularité radiale de $\mathcal{H}$ d'ordre $n-1$ telle que la droite $\check{s}$ duale de $s$ ne soit pas contenue dans $\G_{\mathcal{H}}(\ItrH).$ D'après \cite[Proposition~3.3]{MP13}, au voisinage de tout point générique $m$ de $\check{s},$ le tissu $\Leg\mathcal{H}$ peut se décomposer comme le produit $\W_{n}\boxtimes\W_{d-n},$ où $\W_{n}$ est un $n$-tissu irréductible laissant $\check{s}$ invariante et $\W_{d-n}$ est un $(d-n)$-tissu transverse à $\check{s}.$ De plus, la condition $\check{s}\not\subset\G_{\mathcal{H}}(\ItrH)$ nous assure que le tissu $\W_{d-n}$ est régulier au voisinage de $m.$ Par conséquent $K(\Leg\mathcal{H})$ est holomorphe au voisinage de $m,$ en vertu de \cite[Proposition~2.6]{MP13}.
\end{proof}

\begin{cor}
{\sl Soit $\mathcal{H}$ un feuilletage homogène convexe de degré $d$ sur le plan projectif. Alors le $d$-tissu $\Leg\mathcal{H}$ est plat.}
\end{cor}

\noindent Le théorème suivant est un critère effectif d'holomorphie de la courbure (du tissu dual d'un feuilletage homogène) le long de l'image par l'application de \textsc{Gauss} d'une droite d'inflexion transverse simple, {\it i.e.} d'ordre d'inflexion minimal.

\begin{thm}\label{thm:Barycentre}
{\sl Soit $\mathcal{H}$ un feuilletage homogène de degré $d\geq3$ sur $\pp$ défini par la $1$-forme
$$\omega=A(x,y)\mathrm{d}x+B(x,y)\mathrm{d}y,\quad A,B\in\mathbb{C}[x,y]_d,\hspace{2mm}\pgcd(A,B)=1.$$
Supposons que $\mathcal{H}$ possède une droite d'inflexion $T=(ax+by=0)$ transverse et simple. Supposons en outre que $[-a:b]\in\mathbb{P}^{1}_{\mathbb{C}}$ soit le seul point critique de $\Gunderline$ dans sa fibre $\Gunderline^{-1}(\Gunderline([-a:b])).$ Posons $T'=\mathcal{G}_{\mathcal{H}}(T)$ et considérons la courbe $\Gamma_{(a,b)}$ de $\pp$ définie par
\begin{align*}
Q(x,y;a,b):=\left| \begin{array}{cc}
\dfrac{\partial{P}}{\partial{x}} &  A(b,-a)
\vspace{2mm}
\\
\dfrac{\partial{P}}{\partial{y}} &  B(b,-a)
\end{array} \right|=0,
\quad\text{où}\quad
P(x,y;a,b):=\frac{1}{(ax+by)^{2}}
\left| \begin{array}{cc}
A(x,y)  &  A(b,-a)
\\
B(x,y)  &  B(b,-a)
\end{array} \right|.
\end{align*}
Alors la courbure de $\mathrm{Leg}\mathcal{H}$ est holomorphe sur $T'$ si et seulement si $T=\{ax+by=0\}\subset\Gamma_{(a,b)},$ {\it i.e.} si et seulement si $Q(b,-a\hspace{0.2mm};a,b)=0.$
}
\end{thm}

\begin{rem}\label{rem:Q(b,-a,a,b)}
L'hypothèse que $T=(ax+by=0)$ est une droite d'inflexion pour $\mathcal{H}$ implique que $P\in\mathbb{C}[x,y]_{d-2}$ et donc $Q\in\mathbb{C}[x,y]_{d-3}.$ En particulier lorsque $d=3$ on a $$Q(b,-a\hspace{0.2mm};a,b)=\frac{\Cinv\left(B(b,-a),-A(b,-a)\right)}{\left(\Cinv(b,-a)\right)^2}\hspace{1mm};$$ en effet si on pose $\tilde{a}=A(b,-a)$, $\tilde{b}=B(b,-a)$ et $P(x,y;a,b)=f(a,b)x+g(a,b)y$ on obtient
\begin{align*}
Q(b,-a\hspace{0.2mm};a,b)=f(a,b)\tilde{b}-g(a,b)\tilde{a}=P(\tilde{b},-\tilde{a}\hspace{0.2mm};a,b)
=\frac{\tilde{b}A(\tilde{b},-\tilde{a})-\tilde{a}B(\tilde{b},-\tilde{a})}{(a\tilde{b}-b\tilde{a})^2}
=\frac{\Cinv\left(\tilde{b},-\tilde{a}\right)}{\left(\Cinv(b,-a)\right)^2}.
\end{align*}
\end{rem}

\begin{proof}[\sl D\'emonstration]
\`{A} isomorphisme linéaire près on peut se ramener à $T=(y=rx)$; si $(p,q)$ est la carte affine de $\pd$ associée à la droite $\{y=px-q\}\subset{\mathbb{P}^{2}_{\mathbb{C}}},$ alors $T'=(p=\Gunderline(r))$ avec $\Gunderline(z)=-\dfrac{A(1,z)}{B(1,z)}.$ Comme l'indice de ramification de $\Gunderline$ en $z=r$ est égal à $2$ et comme $z=r$ est l'unique point critique dans sa fibre $\Gunderline^{-1}(\Gunderline(r))$, cette fibre est formée de $d-1$ points distincts, soit $\Gunderline^{-1}(\Gunderline(r))=\{r,z_{1},z_{2},\ldots,z_{d-2}\}.$ De plus, au voisinage de tout point générique de $T'$, le tissu dual de $\mathcal{H}$ se décompose en $\Leg\mathcal{H}=\W_{2}\boxtimes\W_{d-2}$ avec
\begin{align*}
& \W_{2}\Big|_{T'}=\left(\mathrm{d}q-x_{0}(q)\mathrm{d}p\right)^{2}
\qquad\text{et}\qquad
\W_{d-2}\Big|_{T'}=\prod_{i=1}^{d-2}\left(\mathrm{d}q-x_{i}(q)\mathrm{d}p\right),
\end{align*}
où $x_{0}(q)=\dfrac{q}{\Gunderline(r)-r}$\, et\, $x_{i}(q)=\dfrac{q}{\Gunderline(r)-z_{i}},\,i=1,2,\ldots,d-2.$ D'après \cite[Théorème~1]{MP13}, $K(\Leg\mathcal{H})$ est holomorphe le long de $T'$ si et seulement si $T'$ est invariante par le barycentre de $\W_{d-2}$ par rapport à $\W_{2}.$ Or la restriction de $\beta_{\W_{2}}(\W_{d-2})$ à $T'$ est donnée par $\mathrm{d}q-\beta(q)\mathrm{d}p=0$ avec $$\beta=x_{0}+\dfrac{1}{\frac{1}{d-2}\sum\limits_{i=1}^{d-2}\dfrac{1}{x_{i}-x_{0}}}.$$
\noindent Ainsi la courbure de $\Leg\mathcal{H}$ est holomorphe sur $T'$ si et seulement si $\beta=\infty,$ {\it i.e.} si et seulement si $\sum_{i=1}^{d-2}\dfrac{1}{x_{i}-x_{0}}=0,$ car $x_{0}\neq\infty$ ($z=r$ est non fixe par $\Gunderline$). Cette dernière condition se réécrit
\begin{equation}\label{equa:Barycentre-implicite}
\hspace{0.8cm}0=\sum\limits_{i=1}^{d-2}\frac{\Gunderline(r)-z_{i}}{r-z_{i}}=d-2+\left(\Gunderline(r)-r\right)\sum\limits_{i=1}^{d-2}\frac{1}{r-z_{i}}.
\end{equation}
D'autre part les $z_{i}$ sont exactement les racines du polynôme $$F(z):=\dfrac{P(1,z\hspace{0.2mm};-r,1)}{B(1,r)}=\dfrac{A(1,z)+\Gunderline(r)B(1,z)}{(z-r)^{2}}$$ et donc
\begin{align*}
\sum\limits_{i=1}^{d-2}\frac{1}{r-z_{i}}
\hspace{0.3mm}=\hspace{0.3mm}
\sum\limits_{i=1}^{d-2}\left(\frac{1}{F(r)}\prod\limits_{\underset{j\neq i}{j=1}}^{d-2}(r-z_{j})\right)
\hspace{0.3mm}=\hspace{0.3mm}
\frac{1}{F(r)}\sum\limits_{i=1}^{d-2}\hspace{0.5mm}\prod\limits_{\underset{j\neq i}{j=1}}^{d-2}(r-z_{j})
\hspace{0.3mm}=\hspace{0.3mm}
\frac{F'(r)}{F(r)}.
\end{align*}
Ainsi l'équation (\ref{equa:Barycentre-implicite}) est équivalente à $(\Gunderline(r)-r)F'(r)+(d-2)F(r)=0,$ {\it i.e.} à
\begin{equation}\label{equa:Barycentre-explicite}
\hspace{2.2cm}(d-2)P(1,r\hspace{0.2mm};-r,1)+\left(\Gunderline(r)-r\right)\frac{\partial{P}}{\partial{y}}\Big|_{(x,y)=(1,r)}=0\hspace{1mm};
\end{equation}
comme $P\in\mathbb{C}[x,y]_{d-2}$ on peut réécrire (\ref{equa:Barycentre-explicite}) sous la forme
\begin{align*}
\hspace{2.3cm}\left((d-2)P(x,y\hspace{0.2mm};-r,1)-y\frac{\partial{P}}{\partial{y}}+x\Gunderline(r)\frac{\partial{P}}{\partial{y}}\right)\Big|_{y=rx}=0\hspace{1mm};
\end{align*}
celle-ci peut à son tour s'écrire
\begin{align*}
\hspace{-1.3cm}\left(\frac{\partial{P}}{\partial{x}}+\Gunderline(r)\frac{\partial{P}}{\partial{y}}\right)\Big|_{y=rx}=0,
\end{align*}
en vertu de l'identité d'\textsc{Euler}. Il en résulte que $K(\Leg\mathcal{H})$ est holomorphe le long de $T'$ si et seulement si
\begin{align*}
\left(B(1,r)\frac{\partial{P}}{\partial{x}}-A(1,r)\frac{\partial{P}}{\partial{y}}\right)\Big|_{y=rx}=0.
\end{align*}
\end{proof}

\begin{rem}\label{rem:Gh(Gh(r))=Gh(r)}
En degré $3$ l'équation (\ref{equa:Barycentre-implicite}) s'écrit $\dfrac{\Gunderline(r)-z_{1}}{r-z_{1}}=0$; ainsi la courbure du $3$-tissu $\Leg\mathcal{H}$ est holomorphe sur $T'=(p=\Gunderline(r))$ si et seulement si $\Gunderline(r)=z_{1}$, {\it i.e.} si et seulement si $\Gunderline(\Gunderline(r))=\Gunderline(r).$
\end{rem}

\noindent Le théorème suivant est un critère effectif d'holomorphie de la courbure (du tissu dual d'un feuilletage homogène) le long de l'image par l'application de \textsc{Gauss} d'une droite d'inflexion transverse d'ordre maximal.

\begin{thm}\label{thm:Divergence}
{\sl Soit $\mathcal{H}$ un feuilletage homogène de degré $d\geq3$ sur $\pp$ défini par la $1$-forme
$$\omega=A(x,y)\mathrm{d}x+B(x,y)\mathrm{d}y,\quad A,B\in\mathbb{C}[x,y]_d,\hspace{2mm}\pgcd(A,B)=1.$$
Supposons que $\mathcal{H}$ possède une droite d'inflexion transverse $T$ d'ordre maximal $d-1$ et posons $T'=\mathcal{G}_{\mathcal{H}}(T).$ Alors la courbure de $\mathrm{Leg}\mathcal{H}$ est holomorphe le long de $T'$ si et seulement si la $2$-forme $\mathrm{d}\omega$ s'annule sur la droite $T.$
}
\end{thm}
\noindent La démonstration de ce théorème utilise le lemme technique suivant, qui nous sera aussi utile ultérieurement.

\begin{lem}\label{lem:critique-maximale}
{\sl Soit $f:\mathbb{P}^{1}_{\mathbb{C}}\rightarrow\mathbb{P}^{1}_{\mathbb{C}}$ une application rationnelle de degré $d;f(z)=\dfrac{a(z)}{b(z)}$ avec $a$ et $b$ des polynômes sans facteur commun et $\max(\deg a,\deg b)=d.$ Soit $z_{0}\in\mathbb{C}$ tel que $f(z_0)\neq\infty.$ Alors, $z_0$ est un point critique de $f$ de multiplicité $m-1$ si et seulement s'il existe un polynôme $c\in\mathbb{C}[z]$ de degré $\leq d-m$ vérifiant $c(z_0)\neq0$ et tel que $a(z)=f(z_0)b(z)+c(z)(z-z_0)^{m}.$
}
\end{lem}

\begin{proof}[\sl D\'emonstration] D'après la formule de \textsc{Taylor}, l'assertion $z=z_0$ est un point critique de $f$ de multiplicité $m-1$ se traduit par  $f(z)=f(z_0)+h(z)(z-z_0)^m,$ avec $h(z_0)\neq0.$ Par suite $$a(z)-f(z_0)b(z)=c(z)(z-z_0)^m$$ avec $c(z):=h(z)b(z)$, $c(z_0)\neq0$; comme le membre de gauche est un polynôme en $z$ de degré $\leq d$ celui de droite aussi. On constate alors que la fonction $c(z)$ est polynomiale en $z$ de degré $\leq d-m$, d'où l'énoncé.
\end{proof}

\begin{proof}[\sl D\'emonstration du Théorème~\ref{thm:Divergence}]
On peut se ramener à $T=(y=rx)$; si $(p,q)$ est la carte affine de $\pd$ associée à la droite $\{y=px-q\}\subset{\mathbb{P}^{2}_{\mathbb{C}}},$ alors $T'=(p=\Gunderline(r))$ avec $\Gunderline(z)=-\dfrac{A(1,z)}{B(1,z)}.$ De plus, le $d$-tissu $\Leg\mathcal{H}$ est décrit par $\prod_{i=1}^{d}\check{\omega}_i$, où $\check{\omega}_i=\dfrac{\mathrm{d}q}{q}-\lambda_{i}(p)\mathrm{d}p,$\hspace{1mm} $\lambda_{i}(p)=\dfrac{1}{p-p_{i}(p)}$\hspace{1mm} et \hspace{1mm}$\{p_{i}(p)\}=\Gunderline^{-1}(p).$

\noindent
En appliquant les formules (\ref{equa:eta-rst}) et (\ref{equa:eta})  à $\Leg\mathcal{H}=\W(\check{\omega}_1,\check{\omega}_2,\ldots,\check{\omega}_d)$ on constate que $\eta(\mathrm{Leg}\mathcal{H})$ s'écrit sous la forme
\begin{align*}
\hspace{-4.22cm}\eta(\mathrm{Leg}\mathcal{H})=\alpha(p)\mathrm{d}p+\dfrac{\mathrm{d}q}{q}\sum_{1\le i<j<k\le d}\beta_{ijk}(p),
\end{align*}
avec
\begin{align*}
\beta_{ijk}(p)=
\dfrac{-\lambda_{i}'}{(\lambda_{i}-\lambda_{j})(\lambda_{i}-\lambda_{k})}+
\dfrac{-\lambda_{j}'}{(\lambda_{j}-\lambda_{i})(\lambda_{j}-\lambda_{k})}+
\dfrac{-\lambda_{k}'}{(\lambda_{k}-\lambda_{i})(\lambda_{k}-\lambda_{j})}.
\end{align*}
Comme le point $z=r$ est critique non fixe pour $\Gunderline$ de multiplicité $d-1,$ il existe un isomorphisme analytique $\varphi:(\mathbb{C},0)\rightarrow(\mathbb{C},0)$ tel qu'au voisinage de $T'$ on ait
\begin{align*}
\hspace{-1.5cm}\lambda_{i}(p)=\frac{1}{\Gunderline(r)-r}+\varphi\left(\zeta^{i}\left(p-\Gunderline(r)\right)^{\frac{1}{d}}\right), \quad\text{avec}\hspace{2mm}
\zeta=\mathrm{e}^{2\mathrm{i}\pi/d}.
\end{align*}
Notons que
\begin{align*}
\hspace{1.5cm}\lambda_{i}'(p)=\frac{1}{d}\left(p-\Gunderline(r)\right)^{\frac{1-d}{d}}\left[\zeta^{i}\varphi'(0)+
\zeta^{2i}\varphi''(0)\left(p-\Gunderline(r)\right)^{\frac{1}{d}}+
o\left((p-\Gunderline(r))^{\frac{1}{d}}\right)\right],
\end{align*}
et
\begin{align*}
\hspace{-1.2cm}\lambda_{i}(p)-\lambda_{j}(p)=\left(p-\Gunderline(r)\right)^{\frac{1}{d}}\varphi'(0)(\zeta^{i}-\zeta^{j})+
o\left((p-\Gunderline(r))^{\frac{1}{d}}\right).
\end{align*}
Il s'en suit que
\begin{align*}
\hspace{1.1mm}\beta_{ijk}(p)=\left(p-\Gunderline(r)\right)^{-1-\frac{1}{d}}\tilde{\beta}_{ijk}\left((p-\Gunderline(r))^{\frac{1}{d}}\right),
\quad\hspace{1mm}\text{avec}\hspace{1mm}\quad
\tilde{\beta}_{ijk}(z)\in\mathbb{C}\{z\}.
\end{align*}
En fait, si $\langle i',j',k'\rangle$ désigne trois permutations circulaires de $i,j$ et $k,$ on a
\begin{align*}
\hspace{-3.5cm}\tilde{\beta}_{ijk}(0)=-\frac{1}{d\varphi'(0)}\underbrace{\sum_{\langle i',\,j',\,k'\rangle}\frac{\zeta^{i'}}{(\zeta^{i'}-\zeta^{j'})(\zeta^{i'}-\zeta^{k'})}}_{0}=0,
\end{align*}
et
\begin{align*}
\hspace{-1.8cm}\tilde{\beta}'_{ijk}(0)=\frac{\varphi''(0)}{2d\varphi'(0)^{2}}\underbrace{\sum_{\langle i',\,j',\,k'\rangle}\frac{\zeta^{i'}(\zeta^{j'}+\zeta^{k'})}{(\zeta^{i'}-\zeta^{j'})(\zeta^{i'}-\zeta^{k'})}}_{-1}=-\frac{\varphi''(0)}{2d\varphi'(0)^{2}}.
\end{align*}
En posant $\beta(z):=\sum_{1\le i<j<k\le d}\beta_{ijk}(z)$ \hspace{2mm}et\hspace{2mm} $\tilde{\beta}(z):=\sum_{1\le i<j<k\le d}\tilde{\beta}_{ijk}(z),$ on obtient que
\begin{align*}
\hspace{-4.7cm}\beta(p)=\left(p-\Gunderline(r)\right)^{-1-\frac{1}{d}}\tilde{\beta}\left((p-\Gunderline(r))^{\frac{1}{d}}\right).
\end{align*}
Comme $K(\Leg\mathcal{H})=\mathrm{d}\hspace{0.1mm}\eta(\mathrm{Leg}\mathcal{H})=\dfrac{\beta'(p)}{q}\mathrm{d}p\wedge\mathrm{d}q$ et comme $\beta(p)\in\mathbb{C}\{p-\Gunderline(r)\}\Big[\dfrac{1}{p-\Gunderline(r)}\Big],$ on  déduit que $K(\Leg\mathcal{H})$ est holomorphe le long de $T'=(p=\Gunderline(r))$ si et seulement si $\tilde{\beta}(z)\in\raisebox{0.3mm}{$z$}\hspace{0.15mm}\mathbb{C}\{\raisebox{0.3mm}{$z^{d}$}\}$ satisfait la condition
\begin{align*}
&0=\tilde{\beta}'(0)=\sum_{1\le i<j<k\le d}\tilde{\beta}'_{ijk}(0)=-\binom{{d}}{{3}}\frac{\varphi''(0)}{2d\varphi'(0)^{2}},
\end{align*}
\textit{i.e.} si et seulement si $\varphi''(0)=0.$

\noindent D'après le Lemme~\ref{lem:critique-maximale}, le fait que $z=r$ est un point critique (non fixe) de $\Gunderline$ de multiplicité $d-1$ se traduit par $-A(1,z)=\Gunderline(r)B(1,z)+c(z-r)^{d},$ pour un certain $c\in\mathbb{C}^{*}.$ Par suite
\begin{align*}
\hspace{0.15cm}A(x,y)=-\Gunderline(r)B(x,y)-c(y-rx)^{d} \qquad\text{et}\qquad B(x,y)=b_{0}x^{d}+\sum_{i=1}^{d}b_{i}(y-rx)^{i}x^{d-i}.
\end{align*}
Puisque $b_{0}=B(1,r)\neq0,$ on peut supposer sans perte de généralité que $b_{0}=1.$ Ainsi
\begin{align*}
\mathrm{d}\omega\Big|_{y=rx}=\left(d+b_{1}(\Gunderline(r)-r)\right)x^{d-1}\mathrm{d}x\wedge\mathrm{d}y.
\end{align*}
D'autre part, $\Gunderline(z)=\Gunderline(r)+\dfrac{c(z-r)^{d}}{1+b_{1}(z-r)+\cdots}$ et, pour tout $p\in\mathbb{P}^{1}_{\mathbb{C}}$ suffisamment voisin de $\Gunderline(r),$ l'équation $\Gunderline(z)=p$ est équivalente à
\begin{align*}
\hspace{-1.45cm}\left(p-\Gunderline(r)\right)^{\frac{1}{d}}=\frac{c^{\frac{1}{d}}(z-r)}{\sqrt[d]{1+b_{1}(z-r)+\cdots}}
=c^{\frac{1}{d}}(z-r)\left[1-\frac{1}{d}b_{1}(z-r)+\cdots\right].
\end{align*}
Par suite les $p_{i}(p)\in\Gunderline^{-1}(p)$ s'écrivent
\begin{small}
\begin{align*}
\hspace{-4.22cm}p_{i}(p)=r+\frac{1}{c^{\frac{1}{d}}}\zeta^{i}\left(p-\Gunderline(r)\right)^{\frac{1}{d}}+
\frac{b_{1}}{dc^{\frac{2}{d}}}\zeta^{2i}\left(p-\Gunderline(r)\right)^{\frac{2}{d}}+\cdots
\end{align*}
\end{small}
\hspace{-1mm}et donc
\begin{small}
\begin{align*}
\hspace{0.8cm}p-p_{i}(p)=\left(\Gunderline(r)-r\right)-\frac{1}{c^{\frac{1}{d}}}\zeta^{i}\left(p-\Gunderline(r)\right)^{\frac{1}{d}}-
\frac{b_{1}}{dc^{\frac{2}{d}}}\zeta^{2i}\left(p-\Gunderline(r)\right)^{\frac{2}{d}}+\cdots+\left(p-\Gunderline(r)\right)+\cdots.
\end{align*}
\end{small}
\hspace{-1mm}Par conséquent
\begin{align*}
\hspace{0.65cm}\lambda_{i}(p)=\dfrac{1}{p-p_{i}(p)}=\frac{1}{\Gunderline(r)-r}+\varphi'(0)\zeta^{i}\left(p-\Gunderline(r)\right)^{\frac{1}{d}}
+\frac{\varphi''(0)}{2}\zeta^{2i}\left(p-\Gunderline(r)\right)^{\frac{2}{d}}+\cdots,
\end{align*}
avec
\begin{align*}
\hspace{0.44cm}\varphi'(0)=\frac{1}{c^{\frac{1}{d}}(\Gunderline(r)-r)^{2}}\neq0
\qquad\text{et}\qquad
\varphi''(0)=\frac{2}{dc^{\frac{2}{d}}(\Gunderline(r)-r)^{3}}\left[d+b_{1}(\Gunderline(r)-r)\right],
\end{align*}
ce qui termine la démonstration.
\end{proof}

\noindent Comme conséquence immédiate des Théorèmes \ref{thm:holomo-G(I^tr)}, \ref{thm:Barycentre}, \ref{thm:Divergence} et de la Remarque \ref{rem:Q(b,-a,a,b)} nous obtenons la caractérisation suivante de la platitude de la transformée de \textsc{Legendre} d'un feuilletage homogène de degré $3$ sur le plan projectif.
\begin{cor}\label{cor:platitude-degre-3}
{\sl Soit $\mathcal{H}$ un feuilletage homogène de degré $3$ sur $\pp$ défini par la $1$-forme $$\omega=A(x,y)\mathrm{d}x+B(x,y)\mathrm{d}y,\quad A,B\in\mathbb{C}[x,y]_3,\hspace{2mm}\pgcd(A,B)=1.$$ Alors, le $3$-tissu $\Leg\mathcal{H}$ est plat si et seulement si les deux conditions suivantes sont satisfaites:
\begin{itemize}
\item [\textit{(1)}] pour toute droite d'inflexion de $\mathcal{H}$ transverse et simple $T_1=(ax+by=0),$ la droite d'équation $A(b,-a)x+B(b,-a)y=0$ est invariante par $\mathcal{H};$

\item [\textit{(2)}] pour toute droite d'inflexion de $\mathcal{H}$ transverse et double $T_{2},$ la $2$-forme $\mathrm{d}\omega$ s'annule sur $T_{2}.$
\end{itemize}
En particulier, si le feuilletage $\mathcal{H}$ est convexe alors $\Leg\mathcal{H}$ est plat.
}
\end{cor}

\section{Platitude et feuilletages homogènes de type appartenant à $\Z\left[\mathrm{R}_1,\mathrm{R}_2,\ldots,\mathrm{R}_{d-1},\mathrm{T}_1,\mathrm{T}_{d-1}\right]$}\label{sec:type-min-max}
\bigskip

\noindent Nous nous proposons dans ce paragraphe de décrire certaines feuilletages homogènes de degré $d\geq3$ sur $\pp$, de type appartenant à $\Z\left[\mathrm{R}_1,\mathrm{R}_2,\ldots,\mathrm{R}_{d-1},\mathrm{T}_1,\mathrm{T}_{d-1}\right]$ et dont le $d$-tissu dual est plat. Nous considérons ici un feuilletage homogène $\mathcal{H}$ de degré $d\geq3$ sur $\pp$ défini, en carte affine $(x,y),$ par $$\omega=A(x,y)\mathrm{d}x+B(x,y)\mathrm{d}y,\quad A,B\in\mathbb{C}[x,y]_d,\hspace{2mm}\pgcd(A,B)=1.$$ L'application rationnelle $\Gunderline\hspace{1mm}\colon\mathbb{P}^{1}_{\mathbb{C}}\rightarrow \mathbb{P}^{1}_{\mathbb{C}}$, $\Gunderline(z)=-\dfrac{A(1,z)}{B(1,z)},$ nous sera très utile pour établir les énoncés qui suivent.

\begin{pro}\label{pro:omega1-omega2}
{\sl Si $\deg\mathcal{T}_{\mathcal{H}}=2,$ alors le $d$-tissu $\Leg\mathcal{H}$ est plat si et seulement si $\mathcal{H}$ est linéairement conjugué à l'un des deux feuilletages $\mathcal{H}_{1}^{d}$ et $\mathcal{H}_{2}^{d}$ décrits respectivement par les $1$-formes
\begin{itemize}
\item [\texttt{1. }] $\omega_1^{\hspace{0.2mm}d}=y^d\mathrm{d}x-x^d\mathrm{d}y\hspace{0.5mm};$
\item [\texttt{2. }] $\omega_2^{\hspace{0.2mm}d}=x^d\mathrm{d}x-y^d\mathrm{d}y.$
\end{itemize}
}
\end{pro}

\begin{proof}[\sl D\'emonstration]
L'égalité $\deg\mathcal{T}_{\mathcal{H}}=2$ est réalisée si et seulement si nous sommes dans l'une des situations suivantes
\begin{itemize}
\item [(i)] $\mathcal{T}_{\mathcal{H}}=2\cdot\mathrm{R}_{d-1}\hspace{1mm};$

\item [(ii)] $\mathcal{T}_{\mathcal{H}}=2\cdot\mathrm{T}_{d-1}\hspace{1mm};$

\item [(iii)] $\mathcal{T}_{\mathcal{H}}=1\cdot\mathrm{R}_{d-1}+1\cdot\mathrm{T}_{d-1}.$
\end{itemize}
Commençons par étudier l'éventualité (i). Nous pouvons supposer à conjugaison près que les deux singularités radiales de $\mathcal{H}$ sont $[0:1:0]$ et $[1:0:0]$, ce qui revient à supposer que les points $\infty=[1:0],\,[0:1]\in\mathbb{P}^{1}_{\mathbb{C}}$ sont critiques fixes de $\Gunderline$, de même multiplicité $d-1$. Cela se traduit par le fait que $A(x,y)=ay^d$ et $B(x,y)=bx^d$, avec $ab\neq0,$ en vertu du Lemme~\ref{lem:critique-maximale}. Par suite $\omega=ay^d\mathrm{d}x-(-b)x^d\mathrm{d}y$ et nous pouvons évidemment normaliser les coefficients $a$ et $-b$ à $1.$ Ainsi $\mathcal{H}$ est conjugué au feuilletage $\mathcal{H}_{1}^{d}$ décrit par $\omega_1^{\hspace{0.2mm}d}=y^d\mathrm{d}x-x^d\mathrm{d}y$; le $d$-tissu $\Leg\mathcal{H}_{1}^{d}$ est plat car $\mathcal{H}_{1}^{d}$ est convexe.

\noindent Intéressons-nous à la possibilité (ii). \`{A} isomorphisme linéaire près nous pouvons nous ramener à la situation suivante:
\begin{itemize}
\item [$\bullet$] les points $[0:1],\,[1:1]\in\mathbb{P}^{1}_{\mathbb{C}}$ sont critiques non fixes de $\Gunderline$, de même multiplicité $d-1$;

\item [$\bullet$] $\Gunderline(0)$ et $\Gunderline(1)\neq\infty.$
\end{itemize}
Toujours d'après le Lemme~\ref{lem:critique-maximale}, il existe des constantes $\alpha,\beta\in\mathbb{C}^*$ telles que
\begin{align*}
-A(1,z)=\Gunderline(0)B(1,z)+\alpha\hspace{0.1mm}z^d=\Gunderline(1)B(1,z)+\beta(z-1)^d
\end{align*}
avec $\Gunderline(0)\neq0,\hspace{0.5mm}\Gunderline(1)\neq1$\hspace{0.5mm} et \hspace{0.5mm}$\Gunderline(0)\neq\Gunderline(1).$ L'homogénéité de $A$ et $B$ entraîne alors que
\begin{align*}
\quad \omega=\left(\Gunderline(0)\hspace{0.2mm}s\hspace{0.2mm}(y-x)^d-g(1)\hspace{0.2mm}r\hspace{0.2mm}y^d\right)\mathrm{d}x+\left(ry^d-s(y-x)^d\right)\mathrm{d}y
\end{align*}
avec $r=\dfrac{\alpha}{\Gunderline(1)-\Gunderline(0)}\neq0$\hspace{1mm} et \hspace{1mm}$s=\dfrac{\beta}{\Gunderline(1)-\Gunderline(0)}\neq0.$ D'après les Théorèmes \ref{thm:holomo-G(I^tr)} et \ref{thm:Divergence}, le $d$-tissu $\Leg\mathcal{H}$ est plat si et seulement si $\mathrm{d}\omega$ s'annule sur les deux droites $y(y-x)=0.$ Un calcul immédiat montre que
\begin{align*}
\mathrm{d}\omega\Big|_{y=0}=-sd(\Gunderline(0)-1)x^{d-1}\mathrm{d}x\wedge\mathrm{d}y
\hspace{1mm}\quad\text{et}\hspace{1mm}\quad
\mathrm{d}\omega\Big|_{y=x}=rd\Gunderline(1)x^{d-1}\mathrm{d}x\wedge\mathrm{d}y.
\end{align*}
Ainsi $\Leg\mathcal{H}$ est plat si et seulement si $\Gunderline(0)=1$ et $\Gunderline(1)=0$, auquel cas $$\omega=s(y-x)^d\mathrm{d}x+\left(ry^d-s(y-x)^d\right)\mathrm{d}y\hspace{1mm};$$ quitte à remplacer $\omega$ par $\varphi^*\omega,$ où $\varphi(x,y)=\left(s^{\larger{\frac{-1}{d+1}}}x-r^{\larger{\frac{-1}{d+1}}}y,-r^{\larger{\frac{-1}{d+1}}}y\right),$ on se ramène à $$\omega=\omega_2^{\hspace{0.2mm}d}=x^d\mathrm{d}x-y^d\mathrm{d}y.$$

\noindent Considérons pour finir l'éventualité (iii). Nous pouvons supposer que la singularité radiale de $\mathcal{H}$ est le point $[0:1:0]$ et que la droite d'inflexion transverse de $\mathcal{H}$ est la droite $(y=0)$; $\Gunderline(0)\neq\Gunderline(\infty)=\infty$ car $\Gunderline^{-1}(\Gunderline(0))=\{0\}.$ Un raisonnement analogue à celui du cas précédent conduit à
\begin{align*}
&\omega=-\left(\Gunderline(0)\beta\hspace{0.3mm}x^d+\alpha y^d\right)\mathrm{d}x+\beta\hspace{0.3mm}x^d\mathrm{d}y,
\qquad\text{avec}\qquad \alpha\beta\Gunderline(0)\neq0.
\end{align*}
\noindent La courbure du tissu associé à cette $1$-forme ne peut pas être holomorphe sur $\mathcal{G}_{\mathcal{H}}(\{y=0\})$ car
$$\mathrm{d}\omega\Big|_{y=0}=d\beta\hspace{0.3mm}x^{d-1}\mathrm{d}x\wedge\mathrm{d}y\not\equiv0\hspace{1mm};$$
il en résulte que $\Leg\mathcal{H}$ ne peut pas être plat lorsque $\mathcal{T}_{\mathcal{H}}=1\cdot\mathrm{R}_{d-1}+1\cdot\mathrm{T}_{d-1}.$
\end{proof}

\begin{pro}\label{pro:omega3-omega4}
{\sl Soit $\nu$ un entier compris entre $1$ et $d-2$. Si le feuilletage $\mathcal{H}$ est de type
$$
\mathcal{T}_{\mathcal{H}}=1\cdot\mathrm{R}_{\nu}+1\cdot\mathrm{R}_{d-\nu-1}+1\cdot\mathrm{R}_{d-1},
\qquad\text{resp}.\hspace{1.5mm}
\mathcal{T}_{\mathcal{H}}=1\cdot\mathrm{R}_{\nu}+1\cdot\mathrm{R}_{d-\nu-1}+1\cdot\mathrm{T}_{d-1},
$$
alors le $d$-tissu $\Leg\mathcal{H}$ est plat si et seulement si $\mathcal{H}$ est linéairement conjugué au feuilletage $\mathcal{H}_{3}^{d,\nu}$, resp. $\mathcal{H}_{4}^{d,\nu}$ donné par
\[
\hspace{-2cm}\omega_{3}^{\hspace{0.2mm}d,\nu}=\sum\limits_{i=\nu+1}^{d}\binom{{d}}{{i}}x^{d-i}y^i\mathrm{d}x-
\sum\limits_{i=0}^{\nu}\binom{{d}}{{i}}x^{d-i}y^i\mathrm{d}y,
\]
\[
\hspace{-0.9cm}\text{resp}.\hspace{1.5mm}
\omega_{4}^{\hspace{0.2mm}d,\nu}=(d-\nu-1)\sum\limits_{i=\nu+1}^{d}\binom{{d}}{{i}}x^{d-i}y^i\mathrm{d}x+
\nu\sum\limits_{i=0}^{\nu}\binom{{d}}{{i}}x^{d-i}y^i\mathrm{d}y.
\]
}
\end{pro}

\begin{proof}[\sl D\'emonstration]
Dans les deux cas, nous pouvons supposer à conjugaison linéaire près que les points $[0:1],\,[1:0],\,[-1:1]\in\mathbb{P}^{1}_{\mathbb{C}}$ sont critiques de $\Gunderline$, de multiplicité $\nu,$ $d-\nu-1,$ $d-1$ respectivement. Les points $[0:1]$ et $[1:0]$ sont évidemment fixes par $\Gunderline$; le feuilletage $\mathcal{H}$ est de type $\mathcal{T}_{\mathcal{H}}=1\cdot\mathrm{R}_{\nu}+1\cdot\mathrm{R}_{d-\nu-1}+1\cdot\mathrm{R}_{d-1}$ (resp.\hspace{1.5mm}$\mathcal{T}_{\mathcal{H}}=1\cdot\mathrm{R}_{\nu}+1\cdot\mathrm{R}_{d-\nu-1}+1\cdot\mathrm{T}_{d-1}$) si et seulement si le point $[-1:1]$ est fixe (resp. non fixe) par $\Gunderline.$ Puisque $\Gunderline^{-1}(\Gunderline(-1))=\{-1\}$ nous avons $\Gunderline(-1)\neq\Gunderline(\infty)=\infty$. Donc, d'après le Lemme~\ref{lem:critique-maximale}, il existe une constante $\alpha\in\mathbb{C}^*$ et un polynôme homogène $B_{\nu}\in\mathbb{C}[x,y]_\nu$ tels que
\begin{align*}
-A(x,y)=\Gunderline(-1)B(x,y)+\alpha(y+x)^d, \quad B(x,y)=x^{d-\nu}B_{\nu}(x,y)
\quad\text{et}\quad
y^{\nu+1}\hspace{0.5mm}\text{divise}\hspace{1.3mm}A(x,y).
\end{align*}
Il en résulte que
\begin{eqnarray*}
-A(x,y)&=&\Gunderline(-1)x^{d-\nu}B_{\nu}(x,y)+\alpha\sum_{i=0}^{d}\binom{{d}}{{i}}x^{d-\nu}y^{i}
\nonumber\\
&=&\Gunderline(-1)x^{d-\nu}B_{\nu}(x,y)+\alpha\sum_{i=0}^{\nu}\binom{{d}}{{i}}x^{d-\nu}y^{i}+\alpha\sum_{i=\nu+1}^{d}\binom{{d}}{{i}}x^{d-\nu}y^{i}
\nonumber\hspace{1mm};
\end{eqnarray*}
par suite $A(x,y)$ est divisible par $y^{\nu+1}$ si et seulement si
\begin{align*}
-A(x,y)=\alpha\sum_{i=\nu+1}^{d}\binom{{d}}{{i}}x^{d-\nu}y^{i}
\hspace{1mm}\quad\text{et}\hspace{1mm}\quad
\Gunderline(-1)x^{d-\nu}B_{\nu}(x,y)+\alpha\sum_{i=0}^{\nu}\binom{{d}}{{i}}x^{d-\nu}y^{i}=0.
\end{align*}
Quitte à remplacer $\omega=A(x,y)\mathrm{d}x+B(x,y)\mathrm{d}y$ par\hspace{0.2mm} $-\dfrac{1}{\alpha}\omega\hspace{0.5mm}$ on se ramène à
\begin{align*}
& \omega=\sum_{i=\nu+1}^{d}\binom{{d}}{{i}}x^{d-\nu}y^{i}\mathrm{d}x+\frac{1}{\Gunderline(-1)}\sum_{i=0}^{\nu}\binom{{d}}{{i}}x^{d-\nu}y^{i}\mathrm{d}y,
\hspace{1mm}\quad \Gunderline(-1)\neq\Gunderline(0)=0.
\end{align*}
\begin{itemize}
  \item [$\bullet$] Si $\Gunderline(-1)=-1$ nous obtenons le feuilletage $\mathcal{H}_{3}^{d,\nu}$ décrit par
$$\omega_{3}^{\hspace{0.2mm}d,\nu}=\sum\limits_{i=\nu+1}^{d}\binom{{d}}{{i}}x^{d-i}y^i\mathrm{d}x-
\sum\limits_{i=0}^{\nu}\binom{{d}}{{i}}x^{d-i}y^i\mathrm{d}y\hspace{1mm};$$
la transformée de \textsc{Legendre} $\Leg\mathcal{H}_{3}^{d,\nu}$ est plate car $\mathcal{H}_{3}^{d,\nu}$ est convexe.
  \item [$\bullet$] Si $\Gunderline(-1)\neq-1$ alors, d'après les Théorèmes \ref{thm:holomo-G(I^tr)} et \ref{thm:Divergence}, le $d$-tissu $\Leg\mathcal{H}$ est plat si et seulement si
$$\hspace{1.5cm}0\equiv\mathrm{d}\omega\Big|_{y=-x}=
\binom{{d}}{{\nu+1}}\dfrac{(-1)^{\nu+1}(\nu+1)}{(d-1)\Gunderline(-1)}\left[\Gunderline(-1)\nu-d+\nu+1\right]x^{d-1}\mathrm{d}x\wedge\mathrm{d}y,$$
{\it i.e.} si et seulement si $\Gunderline(-1)=\dfrac{d-\nu-1}{\nu},$ auquel cas
$$(d-\nu-1)\omega=\omega_{4}^{\hspace{0.2mm}d,\nu}=(d-\nu-1)\sum\limits_{i=\nu+1}^{d}\binom{{d}}{{i}}x^{d-i}y^i\mathrm{d}x+
\nu\sum\limits_{i=0}^{\nu}\binom{{d}}{{i}}x^{d-i}y^i\mathrm{d}y.$$
\end{itemize}
\end{proof}

\begin{pro}\label{pro:omega5-omega6}
{\sl Si le feuilletage $\mathcal{H}$ est de type
$$
\mathcal{T}_{\mathcal{H}}=1\cdot\mathrm{R}_{d-2}+1\cdot\mathrm{T}_{1}+1\cdot\mathrm{R}_{d-1},
\qquad\text{resp}.\hspace{1.5mm}
\mathcal{T}_{\mathcal{H}}=1\cdot\mathrm{R}_{d-2}+1\cdot\mathrm{T}_{1}+1\cdot\mathrm{T}_{d-1},
$$
alors le $d$-tissu $\Leg\mathcal{H}$ est plat si et seulement si $\mathcal{H}$ est linéairement conjugué au feuilletage $\mathcal{H}_{5}^{d}$, resp. $\mathcal{H}_{6}^{d}$ décrit par
\[
\hspace{-4.2cm}\omega_{5}^{\hspace{0.2mm}d}=2y^d\mathrm{d}x+x^{d-1}(yd-(d-1)x)\mathrm{d}y,
\]
\[
\hspace{1cm}\text{resp}.\hspace{1.5mm}
\omega_{6}^{\hspace{0.2mm}d}=\left((d-1)^2x^d-d(d-1)x^{d-1}y+(d+1)y^d\right)\mathrm{d}x+x^{d-1}\left(yd-(d-1)x\right)\mathrm{d}y.
\]
}
\end{pro}
\begin{proof}[\sl D\'emonstration]
Nous allons traiter ces deux types simultanément. \`{A} isomorphisme linéaire près, nous pouvons nous ramener à la situation suivante: les points $[1:0],\,[1:1],\,[0:1]\in\mathbb{P}^{1}_{\mathbb{C}}$ sont critiques de $\Gunderline$, de multiplicité $d-2,$ $1,$ $d-1$ respectivement. Le point $[1:0]$ (resp.  $[1:1]$) est fixe (resp. non fixe) par $\Gunderline$; le feuilletage $\mathcal{H}$ est de type $\mathcal{T}_{\mathcal{H}}=1\cdot\mathrm{R}_{d-2}+1\cdot\mathrm{T}_{1}+1\cdot\mathrm{R}_{d-1}$ (resp.\hspace{1.5mm}$\mathcal{T}_{\mathcal{H}}=1\cdot\mathrm{R}_{d-2}+1\cdot\mathrm{T}_{1}+1\cdot\mathrm{T}_{d-1}$) si et seulement si le point $[0:1]$ est fixe (resp. non fixe) par $\Gunderline.$ Puisque $\Gunderline^{-1}(\Gunderline(0))=\{0\}$ nous avons $\Gunderline(0)\neq\Gunderline(1)$ et $\Gunderline(0)\neq\Gunderline(\infty)=\infty$; de plus $\Gunderline(1)\neq\Gunderline(\infty)=\infty$ car $\Gunderline^{-1}(\Gunderline(\infty))=\{\infty,z_{0}\}$ pour un certain point $z_{0}\neq\infty,$ non critique de $\Gunderline.$ Par suite, d'après le Lemme~\ref{lem:critique-maximale}, il existe une constante $\alpha\in\mathbb{C}^*$ telle que
\begin{align*}
-A(x,y)=\Gunderline(0)B(x,y)+\alpha y^d
\qquad\text{et}\qquad
B(x,y)=\frac{\raisebox{-0.4mm}{$\alpha$}}{\raisebox{0.4mm}{$s$}}x^{d-1}\left(yd-(d-1)x\right),
\end{align*}
avec $s=\Gunderline(1)-\Gunderline(0)\neq0.$
Quitte à multiplier $\omega=A(x,y)\mathrm{d}x+B(x,y)\mathrm{d}y$ par $\dfrac{\raisebox{-0.4mm}{$s$}}{\raisebox{0.4mm}{$\alpha$}}$ on se ramène à
\begin{align*}
& \omega=-\left(\Gunderline(0)\,x^{d-1}\left(yd-(d-1)x\right)+sy^d\right)\mathrm{d}x+x^{d-1}\left(yd-(d-1)x\right)\mathrm{d}y.
\end{align*}
D'après ce qui précède le point $[1:1]$ est le seul point critique de $\Gunderline$ dans sa fibre $\Gunderline^{-1}(\Gunderline(1)).$ Donc, d'après le Théorème~\ref{thm:Barycentre}, la courbure de $\Leg\mathcal{H}$ est holomorphe sur $\mathcal{G}_{\mathcal{H}}(\{y=x\})$ si et seulement si $$\hspace{2cm}0=Q(1,1;-1,1)=-\frac{1}{6}sd(d-1)(d-2)(\Gunderline(0)+s+2),$$ {\it i.e.} si et seulement si $s=-\Gunderline(0)-2.$
\begin{itemize}
  \item [$\bullet$] Si $\Gunderline(0)=0$ alors la condition $s=-\Gunderline(0)-2=-2$ est suffisante pour que $\Leg\mathcal{H}$ soit plat, en vertu du Théorème~\ref{thm:holomo-G(I^tr)}, auquel cas $$\omega=\omega_{5}^{\hspace{0.2mm}d}=2y^d\mathrm{d}x+x^{d-1}(yd-(d-1)x)\mathrm{d}y.$$

  \item [$\bullet$] Si $\Gunderline(0)\neq0$ alors, d'après les Théorèmes \ref{thm:holomo-G(I^tr)} et \ref{thm:Divergence}, $\Leg\mathcal{H}$ est plat si et seulement si
      $$
      \hspace{1cm}
      s=-\Gunderline(0)-2
      \quad\hspace{1cm} \text{et} \quad\hspace{1cm}
      0\equiv\mathrm{d}\omega\Big|_{y=0}=d(\Gunderline(0)-d+1)x^{d-1}\mathrm{d}x\wedge\mathrm{d}y,
      $$
{\it i.e.} si et seulement si $\Gunderline(0)=d-1$\, et \,$s=-d-1$, auquel cas
$$\hspace{1cm}\omega=\omega_{6}^{\hspace{0.2mm}d}=\left((d-1)^2x^d-d(d-1)x^{d-1}y+(d+1)y^d\right)\mathrm{d}x+x^{d-1}\left(yd-(d-1)x\right)\mathrm{d}y.$$
\end{itemize}
\end{proof}

\section{Classification des feuilletages homogènes de degré trois à transformée de \textsc{Legendre} plate}
\bigskip

\noindent Dans ce paragraphe nous allons classifier, à automorphisme de $\pp$ près, les feuilletages homogènes de degré $3$ sur le plan projectif dont le $3$-tissu dual est plat. Plus précisément nous allons démontrer le théorème suivant.

\begin{thm}\label{thm:Class-Homog3-Plat}
{\sl Soit $\mathcal{H}$ un feuilletage homogène de degré $3$ sur le plan projectif $\hspace{0.1mm}\pp$. Alors le $3$-tissu dual $\mathrm{Leg}\mathcal{H}$ de $\mathcal{H}$ est plat si et seulement si $\mathcal{H}$ est linéairement conjugué à l'un des onze feuilletages $\mathcal{H}_{1},\ldots,\mathcal{H}_{11}$ décrits respectivement en carte affine par les $1$-formes
\begin{itemize}
\item [\texttt{1. }] \hspace{1mm}$\omega_1\hspace{1mm}=y^3\mathrm{d}x-x^3\mathrm{d}y$;
\smallskip
\item [\texttt{2. }] \hspace{1mm}$\omega_2\hspace{1mm}=x^3\mathrm{d}x-y^3\mathrm{d}y$;
\smallskip
\item [\texttt{3. }] \hspace{1mm}$\omega_3\hspace{1mm}=y^2(3x+y)\mathrm{d}x-x^2(x+3y)\mathrm{d}y$;
\smallskip
\item [\texttt{4. }] \hspace{1mm}$\omega_4\hspace{1mm}=y^2(3x+y)\mathrm{d}x+x^2(x+3y)\mathrm{d}y$;
\smallskip
\item [\texttt{5. }] \hspace{1mm}$\omega_5\hspace{1mm}=2y^3\mathrm{d}x+x^2(3y-2x)\mathrm{d}y$;
\smallskip
\item [\texttt{6. }] \hspace{1mm}$\omega_6\hspace{1mm}=(4x^3-6x^2y+4y^3)\mathrm{d}x+x^2(3y-2x)\mathrm{d}y$;
\smallskip
\item [\texttt{7. }] \hspace{1mm}$\omega_7\hspace{1mm}=y^3\mathrm{d}x+x(3y^2-x^2)\mathrm{d}y$;
\smallskip
\item [\texttt{8. }] \hspace{1mm}$\omega_8\hspace{1mm}=x(x^2-3y^2)\mathrm{d}x-4y^3\mathrm{d}y$;
\smallskip
\item [\texttt{9. }] \hspace{1mm}$\omega_9\hspace{1mm}=y^{2}\left((-3+\mathrm{i}\sqrt{3})x+2y\right)\mathrm{d}x+
                                                       x^{2}\left((1+\mathrm{i}\sqrt{3})x-2\mathrm{i}\sqrt{3}y\right)\mathrm{d}y$;
\smallskip
\item [\texttt{10. }] \hspace{-1mm}$\omega_{10}=(3x+\sqrt{3}y)y^2\mathrm{d}x+(3y-\sqrt{3}x)x^2\mathrm{d}y$;
\smallskip
\item [\texttt{11. }]  \hspace{-1mm}$\omega_{11}=(3x^3+3\sqrt{3}x^2y+3xy^2+\sqrt{3}y^3)\mathrm{d}x+(\sqrt{3}x^3+3x^2y+3\sqrt{3}xy^2+3y^3)\mathrm{d}y$.
\end{itemize}
}
\end{thm}

\noindent Considérons un feuilletage homogène $\mathcal{H}$ de degré $3$ sur $\pp$ défini, en carte affine $(x,y)$, par
$$\omega=A(x,y)\mathrm{d}x+B(x,y)\mathrm{d}y,$$
où $A$ et $B$ désignent des polynômes homogènes de degré $3$ sans composante commune; la classification menant au Théorème~\ref{thm:Class-Homog3-Plat} est établie au cas par cas suivant que $\deg\mathcal{T}_{\mathcal{H}}=2,3$ ou $4$, {\it i.e.} suivant la nature du support du diviseur $\Dtr$ qui peut être deux droites, trois droites ou quatre droites. Pour ce faire commençons par établir les deux lemmes suivants.

\begin{lem}\label{lem:2T1+1TR2}
{\sl Si $\mathcal{T}_{\mathcal{H}}=2\cdot\mathrm{T}_{1}+1\cdot\mathrm{R}_{2}$, resp. $\mathcal{T}_{\mathcal{H}}=2\cdot\mathrm{T}_{1}+1\cdot\mathrm{T}_{2}$, alors, à conjugaison linéaire près, la $1$-forme $\omega$ décrivant $\mathcal{H}$ est du type
\[
\hspace{-3.6cm}\omega=y^3\mathrm{d}x+\left(\beta\,x^3-3\beta\,xy^2+\alpha\,y^3\right)\mathrm{d}y,\qquad \beta\left((2\beta-1)^2-\alpha^2\right)\neq0,
\]
\[
\text{resp}.\hspace{1.5mm}\omega=\left(x^3-3xy^2+\alpha\,y^3\right)\mathrm{d}x+\left(\delta\,x^3-3\delta\,xy^2+\beta\,y^3\right)\mathrm{d}y,\qquad
(\beta-\alpha\delta)\left((\beta-2)^2-(\alpha-2\delta)^2\right)\neq0.
\]
}
\end{lem}

\begin{proof}[\sl D\'emonstration] \`{A} isomorphisme près nous pouvons nous ramener à $\Dtr=cy^{2}(y-x)(y+x)$ pour un certain $c\in \mathbb{C}^*$. Le produit $\Cinv(1,1)\Cinv(1,-1)$ est évidemment non nul; $\mathcal{H}$ est de type $\mathcal{T}_{\mathcal{H}}=2\cdot\mathrm{T}_{1}+1\cdot\mathrm{R}_{2}$ (resp. $\mathcal{T}_{\mathcal{H}}=2\cdot\mathrm{T}_{1}+1\cdot\mathrm{T}_{2}$) si et seulement si $\Cinv(1,0)=0$ (resp. $\Cinv(1,0)\neq0$).
Écrivons les coefficients $A$ et $B$ de $\omega$ sous la forme
$$A(x,y)=a_{0}x^3+a_{1}x^2y+a_{2}xy^2+a_{3}y^3\qquad \text{et} \qquad B(x,y)=b_{0}x^3+b_{1}x^2y+b_{2}xy^2+b_{3}y^3\hspace{1mm};$$
nous avons donc
\begin{align*}
\hspace{-1.3cm}\Cinv=a_0x^4+(a_1+b_0)x^3y+(a_2+b_1)x^2y^2+(a_3+b_2)xy^3+b_3y^4
\end{align*}
et
\begin{align*}
&\hspace{1.2cm}\Dtr=(a_0b_1-a_1b_0)x^4+2(a_0b_2-a_2b_0)x^3y+(3a_0b_3+a_1b_2-a_2b_1-3a_3b_0)x^2y^2\\
&\hspace{2.1cm}+2(a_1b_3-a_3b_1)xy^3+(a_2b_3-a_3b_2)y^4.
\end{align*}
Ainsi $\Cinv(1,0)=a_0$\, et

\begin{equation}\label{equa:2T1+1TR2}
\Dtr=cy^{2}(y-x)(y+x)\hspace{4mm}\Leftrightarrow\hspace{4mm}\left\{\begin{array}{lllll}
a_0\hspace{0.2mm}b_1=a_1b_0\\
a_0\hspace{0.2mm}b_2=a_2b_0\\
a_1b_3=a_3b_1\\
a_2b_3-a_3b_2=c\\
3a_0b_3+a_1b_2-a_2b_1-3a_3b_0=-c
\end{array}\right.
\end{equation}

\begin{itemize}
  \item [$\bullet$] Si $a_0\neq0$ alors le système (\ref{equa:2T1+1TR2}) est équivalent à
$$a_1=0,\qquad a_2=-3a_0,\qquad b_1=0,\qquad b_2=-3b_0,\qquad c=-3(a_0b_3-a_3b_0).$$
Posons $a_3=a_0\alpha,\quad b_0=a_0\delta,\quad b_3=a_0\beta$; alors, quitte à diviser $\omega$ par $a_0$, cette forme s'écrit
$$\omega=\left(x^3-3xy^2+\alpha\,y^3\right)\mathrm{d}x+\left(\delta\,x^3-3\delta\,xy^2+\beta\,y^3\right)\mathrm{d}y\hspace{1mm};$$
un calcul direct montre que la condition $c\,\Cinv(1,1)\Cinv(1,-1)\neq0$ est vérifiée si et seulement si $(\beta-\alpha\delta)\left((\beta-2)^2-(\alpha-2\delta)^2\right)\neq0.$
  \item [$\bullet$] Si $a_0=0$ alors le système (\ref{equa:2T1+1TR2}) conduit à
  $$a_1=a_2=b_1=0,\quad b_2=-3b_0,\quad c=3a_3b_0\neq0.$$ Écrivons $b_0=a_3\beta$\, et \,$b_3=a_3\alpha$; alors, quitte à remplacer $\omega$ par $\dfrac{1}{a_3}\omega,$ on se ramène à $$\omega=y^3\mathrm{d}x+\left(\beta\,x^3-3\beta\,xy^2+\alpha\,y^3\right)\mathrm{d}y,$$ et la non nullité du produit $c\,\Cinv(1,1)\Cinv(1,-1)$ est équivalente à $\beta\left((2\beta-1)^2-\alpha^2\right)\neq0.$
\end{itemize}
\end{proof}

\begin{lem}\label{lem:deg-type=4}
{\sl Si le diviseur $\Dtr$ est réduit, {\it i.e.} si $\deg\mathcal{T}_\mathcal{H}=4$, alors $\omega$ est, à conjugaison linéaire près, de l'une des formes suivantes
\begin{itemize}
\item [\texttt{1.}] $y^2\left((2r+3)x-(r+2)y\right)\mathrm{d}x-x^2(x+ry)\mathrm{d}y$,\\
                    \hspace{1cm} \text{où}\hspace{1mm}
                    \begin{small}$r(r+1)(r+2)(r+3)(2r+3)\neq0$\end{small};
                    \smallskip
\item [\texttt{2.}] $s\hspace{0.12mm}y^2\left((2r+3)x-(r+2)y\right)\mathrm{d}x-x^2(x+ry)\mathrm{d}y$,\\
                    \hspace{1cm} \text{où}\hspace{1mm}
                    \begin{small}$rs(s-1)(r+1)(r+2)(r+3)(2r+3)\left(s(2r+3)^2-r^2\right)\neq0$\end{small};
                    \smallskip
\item [\texttt{3.}] $t\hspace{0.12mm}y^2\left((2r+3)x-(r+2)y\right)\mathrm{d}x-x^2(x+ry)\mathrm{d}(sy-x)$,\\
                    \hspace{1cm} \text{où}\hspace{1mm}
                    \begin{small}$r\hspace{0.17mm}s\hspace{0.17mm}t(r+1)(r+2)(r+3)(2r+3)(s-t-1)(tu^3-r^2su-r^2v)\neq0$\end{small},\quad \begin{small}$u=2r+3$\end{small}
                    \hspace{2mm}\text{et}\hspace{2mm}
                    \begin{small}$v=r(r+2)$\end{small};
                    \smallskip
\item [\texttt{4.}] $uy^2\left((2r+3)x-(r+2)y\right)\mathrm{d}(y-sx)-x^2(x+ry)\mathrm{d}(ty-x)$,\\
                    \hspace{1cm} \text{où}\hspace{1mm}
                    \begin{small}$ur(r+1)(r+2)(r+3)(2r+3)(st-1)(su+t-u-1)(uv^4+suwv^3+r^2twv+r^2w^2)\neq0$\end{small},\\
                    \begin{small}$v=2r+3$\end{small}
                    \hspace{2mm}\text{et}\hspace{2mm}
                    \begin{small}$w=r(r+2)$\end{small}.
\end{itemize}
Ces quatre modèles sont respectivement de types $3\cdot\mathrm{R}_1+1\cdot\mathrm{T}_1,\hspace{1mm}2\cdot\mathrm{R}_1+2\cdot\mathrm{T}_1,\hspace{1mm}1\cdot\mathrm{R}_1+3\cdot\mathrm{T}_1,
\hspace{1mm}4\cdot\mathrm{T}_1.$
}
\end{lem}

\begin{proof}[\sl D\'emonstration]
D'après la Remarque \ref{rem:2d-2-R1} le feuilletage $\mathcal{H}$ ne peut être de type $4\cdot\mathrm{R}_{1};$ nous sommes donc dans l'une des situations suivantes
\begin{itemize}
\item [(i)] $\mathcal{T}_{\mathcal{H}}=3\cdot\mathrm{R}_{1}+1\cdot\mathrm{T}_{1}\hspace{1mm};$

\item [(ii)] $\mathcal{T}_{\mathcal{H}}=2\cdot\mathrm{R}_{1}+2\cdot\mathrm{T}_{1};$

\item [(iii)] $\mathcal{T}_{\mathcal{H}}=1\cdot\mathrm{R}_{1}+3\cdot\mathrm{T}_{1};$

\item [(iv)] $\mathcal{T}_{\mathcal{H}}=4\cdot\mathrm{T}_{1}.$
\end{itemize}
\`{A} conjugaison linéaire près nous pouvons nous ramener à $\Dtr=cxy(y-x)(y-\alpha x)$ pour certains $c,\alpha\in \mathbb{C}^*,\alpha\neq1$. Dans la dernière éventualité nous avons $$\Cinv(0,1)\Cinv(1,0)\Cinv(1,1)\Cinv(1,\alpha)\neq0$$ et dans les cas (i), resp. (ii), resp. (iii) nous pouvons supposer que
\begin{align*}
\left\{
\begin{array}[c]{c}
\Cinv(0,1)=0
\\
\Cinv(1,0)=0
\\
\Cinv(1,1)=0
\\
\Cinv(1,\alpha)\neq0
\end{array}
\right.
\hspace{2mm}
\text{resp.}\hspace{1.5mm}\left\{
\begin{array}[c]{c}
\Cinv(0,1)=0
\\
\Cinv(1,0)=0
\\
\Cinv(1,1)\neq0
\\
\Cinv(1,\alpha)\neq0
\end{array}
\right.
\hspace{2mm}
\text{resp.}\hspace{1.5mm}\left\{
\begin{array}[c]{c}
\Cinv(0,1)=0
\\
\Cinv(1,0)\neq0
\\
\Cinv(1,1)\neq0
\\
\Cinv(1,\alpha)\neq0
\end{array}
\right.
\end{align*}
Comme dans le lemme précédent, en écrivant
$$A(x,y)=a_{0}x^3+a_{1}x^2y+a_{2}xy^2+a_{3}y^3\qquad \text{et} \qquad B(x,y)=b_{0}x^3+b_{1}x^2y+b_{2}xy^2+b_{3}y^3$$
nous obtenons que
\begin{equation}\label{equa:deg-type=4}
\Dtr=cxy(y-x)(y-\alpha x)\hspace{4mm}\Leftrightarrow\hspace{4mm}\left\{\begin{array}{lllll}
a_0\hspace{0.2mm}b_1=a_1b_0\\
a_2b_3=a_3b_2\\
2(a_1b_3-a_3b_1)=c\\
2(a_0b_2-a_2b_0)=c\alpha\\
3a_0b_3+a_1b_2-a_2b_1-3a_3b_0=-c(\alpha+1)
\end{array}\right.
\end{equation}
Envisageons l'éventualité (iv). Comme $c\neq0,$ $a_0=\Cinv(1,0)\neq0$\, et \,$b_3=\Cinv(0,1)\neq0$, le système (\ref{equa:deg-type=4}) est équivalent à
\begin{equation*}
\left\{
\begin{array}{lllll}
b_1=\dfrac{a_1b_0}{a_0}\\
a_2=\dfrac{a_3b_2}{b_3}\\
c=\dfrac{2a_1(a_0b_3-a_3b_0)}{a_0}\\
a_0\hspace{0.2mm}b_2-\alpha a_1b_3=0\\
(3a_0+2\alpha a_1+2a_1)b_3+a_1b_2=0
\end{array}
\right.
\hspace{4mm}\Leftrightarrow\hspace{4mm}
\left\{
\begin{array}{lllll}
b_1=\dfrac{a_1b_0}{a_0}\\
a_2=\dfrac{a_3a_1\alpha}{a_0}\\
c=\dfrac{2a_1(a_0b_3-a_3b_0)}{a_0}\\
b_2=\dfrac{a_1b_3\alpha}{a_0}\\
a_1(a_1+2a_0)\alpha+a_0(2a_1+3a_0)=0
\end{array}
\right.
\end{equation*}
Donc $a_1\neq0$ et puisque $\alpha\neq0$, le produit $(a_1+2a_0)(2a_1+3a_0)$ est non nul. Il s'en suit que
\begin{align*}
& a_2=-\dfrac{a_3(2a_1+3a_0)}{a_1+2a_0},&& b_1=\dfrac{a_1b_0}{a_0},&& b_2=-\dfrac{b_3(2a_1+3a_0)}{a_1+2a_0},\\
& \alpha=-\dfrac{a_0(2a_1+3a_0)}{a_1(a_1+2a_0)},&& c=\dfrac{2a_1(a_0b_3-a_3b_0)}{a_0}.
\end{align*}
Posons $r=\dfrac{a_1}{a_0},\quad s=-\dfrac{a_3}{b_3},\quad t=-\dfrac{b_0}{a_0},\quad u=-\dfrac{b_3}{a_1+2a_0}$; alors
\begin{align*}
\hspace{6mm}& b_0=-ta_0,&& b_1=-rta_0,&& b_2=(2r+3)ua_0,&& b_3=-u(r+2)a_0,\\
\hspace{6mm}& a_1=ra_0,&& a_2=-su(2r+3)a_0,&& a_3=su(r+2)a_0,\\
\hspace{6mm}& \alpha=-\frac{2r+3}{r(r+2)},&& c=2r(r+2)u(st-1)a_0^{2}.
\end{align*}
Quitte à remplacer $\omega$ par $\dfrac{1}{a_0}\omega,$ le coefficient $a_0$ vaut $1$ et $\omega$ s'écrit
\begin{eqnarray*}
\omega\hspace{-1mm}&=&\hspace{-1mm}\left(x^3+rx^2y-su(2r+3)xy^2+su(r+2)y^3\right)\mathrm{d}x+\left(-tx^3-rtx^2y+u(2r+3)xy^2-u(r+2)y^3\right)\mathrm{d}y
\\
\hspace{-1mm}&=&\hspace{-1mm}uy^2\left((2r+3)x-(r+2)y\right)\mathrm{d}(y-sx)-x^2(x+ry)\mathrm{d}(ty-x)
\hspace{1mm};
\end{eqnarray*}
un calcul direct montre que la condition $c\alpha(\alpha-1)\Cinv(0,1)\Cinv(1,0)\Cinv(1,1)\Cinv(1,\alpha)\neq0$ est équivalente à
\begin{align*}
ur(r+1)(r+2)(r+3)(2r+3)(st-1)(su+t-u-1)(uv^4+suwv^3+r^2twv+r^2w^2)\neq0
\end{align*}
avec $v=2r+3$\, et \,$w=r(r+2).$

\noindent Maintenant nous étudions la possibilité (iii). Dans ce cas nous avons $b_3=\Cinv(0,1)=0$ et $a_0=\Cinv(1,0)\neq0$; le système (\ref{equa:deg-type=4}) conduit à
\begin{align*}
& a_2=-\dfrac{a_3(2a_1+3a_0)}{a_1+2a_0},&& b_1=\dfrac{a_1b_0}{a_0},&& b_2=0,&& \alpha=-\dfrac{a_0(2a_1+3a_0)}{a_1(a_1+2a_0)},&& c=-\dfrac{2a_1a_3b_0}{a_0}.
\end{align*}
En posant $r=\dfrac{a_1}{a_0},\quad s=-\dfrac{b_0}{a_0}$\hspace{2mm} et \hspace{2mm}$t=-\dfrac{a_3}{a_1+2a_0},$ nous obtenons que
\begin{align*}
\hspace{6mm}& b_0=-sa_0,&& b_1=-rsa_0,&& b_2=b_3=0,&& c=-2rst(r+2)a_0^{2}, \\
\hspace{6mm}& a_1=ra_0,&& a_2=t(2r+3)a_0,&& a_3=-t(r+2)a_0,&& \alpha=-\frac{2r+3}{r(r+2)}.
\end{align*}
Quitte à diviser $\omega$ par $a_0$ on se ramène à
\begin{eqnarray*}
\omega\hspace{-1mm}&=&\hspace{-1mm}\left(x^3+rx^2y+t(2r+3)xy^2-t(r+2)y^3\right)\mathrm{d}x-sx^2(x+ry)\mathrm{d}y
\\
\hspace{-1mm}&=&t\hspace{0.12mm}y^2\left((2r+3)x-(r+2)y\right)\mathrm{d}x-x^2(x+ry)\mathrm{d}(sy-x),
\end{eqnarray*}
et la non nullité du produit $c\alpha(\alpha-1)\Cinv(1,0)\Cinv(1,1)\Cinv(1,\alpha)$ se traduit par
\begin{align*}
r\hspace{0.17mm}s\hspace{0.17mm}t(r+1)(r+2)(r+3)(2r+3)(s-t-1)(tu^3-r^2su-r^2v)\neq0,
\hspace{2mm}\text{avec}\hspace{2mm}
u=2r+3\hspace{2mm} \text{et} \hspace{2mm}v=r(r+2).
\end{align*}
\noindent Les deux premiers cas se traitent de façon analogue.
\end{proof}

\begin{proof}[\sl D\'emonstration du Théorème~\ref{thm:Class-Homog3-Plat}]
\textsl{Premier cas}: $\deg\mathcal{T}_{\mathcal{H}}=2.$ Dans ce cas le $3$-tissu $\Leg\mathcal{H}$ est plat si et seulement si la $1$-forme $\omega$ définissant $\mathcal{H}$ est linéairement conjuguée à l'une des deux $1$-formes
$$\omega_1=y^3\mathrm{d}x-x^3\mathrm{d}y\qquad\hspace{1.5mm} \text{et} \qquad\hspace{1.5mm} \omega_2=x^3\mathrm{d}x-y^3\mathrm{d}y.$$
C'est une application directe de la Proposition~\ref{pro:omega1-omega2} pour $d=3.$

\noindent\textsl{Second cas}: $\deg\mathcal{T}_{\mathcal{H}}=3.$
\begin{itemize}
  \item [$\bullet$] Si $\mathcal{T}_\mathcal{H}=2\cdot\mathrm{R}_1+1\cdot\mathrm{R}_2$, resp. $\mathcal{T}_\mathcal{H}=2\cdot\mathrm{R}_1+1\cdot\mathrm{T}_2$, alors, d'après la Proposition~\ref{pro:omega3-omega4}, $\Leg\mathcal{H}$ est plat si et seulement si $\omega$ est conjuguée à
      \[
        \hspace{1cm}\omega_{3}^{\hspace{0.2mm}3,1}=\,\sum\limits_{i=2}^{3}\binom{{3}}{{i}}x^{3-i}y^i\mathrm{d}x-
        \sum\limits_{i=0}^{1}\binom{{3}}{{i}}x^{3-i}y^i\mathrm{d}y\,=\,y^2(3x+y)\mathrm{d}x-x^2(x+3y)\mathrm{d}y\,=\,\omega_3,
      \]
      \[
        \text{resp}.\hspace{1.5mm}
        \omega_{4}^{\hspace{0.2mm}3,1}=\,\sum\limits_{i=2}^{3}\binom{{3}}{{i}}x^{3-i}y^i\mathrm{d}x+
        \sum\limits_{i=0}^{1}\binom{{3}}{{i}}x^{3-i}y^i\mathrm{d}y\,=\,y^2(3x+y)\mathrm{d}x+x^2(x+3y)\mathrm{d}y\,=\,\omega_4.
      \]
  \item [$\bullet$]  Si $\mathcal{T}_\mathcal{H}=1\cdot\mathrm{R}_1+1\cdot\mathrm{T}_1+1\cdot\mathrm{R}_2$, resp. $\mathcal{T}_\mathcal{H}=1\cdot\mathrm{R}_1+1\cdot\mathrm{T}_1+1\cdot\mathrm{T}_2$, alors, d'après la Proposition~\ref{pro:omega5-omega6}, $\Leg\mathcal{H}$ est plat si et seulement si $\omega$ est conjuguée à
      \[
        \hspace{-1.5cm}\omega_{5}^{\hspace{0.2mm}3}\,=\,2y^3\mathrm{d}x+x^2(3y-2x)\mathrm{d}y\,=\,\omega_5,
      \]
      \[
        \text{resp}.\hspace{1.5mm}\omega_{6}^{\hspace{0.2mm}3}\,=\,(4x^3-6x^2y+4y^3)\mathrm{d}x+x^2(3y-2x)\mathrm{d}y\,=\,\omega_6.
      \]
  \item [$\bullet$] Si $\mathcal{T}_\mathcal{H}=2\cdot\mathrm{T}_1+1\cdot\mathrm{R}_2$, alors, d'après le Lemme~\ref{lem:2T1+1TR2}, la $1$-forme $\omega$ est du type
      \[
       \omega=y^3\mathrm{d}x+\left(\beta\,x^3-3\beta\,xy^2+\alpha\,y^3\right)\mathrm{d}y,\qquad \beta\left((2\beta-1)^2-\alpha^2\right)\neq0,
      \]
      et dans ce cas nous avons $\ItrH=(y-x)(y+x)$. D'après le Corollaire~\ref{cor:platitude-degre-3}, le $3$-tissu $\Leg\mathcal{H}$ est plat si et seulement si $$\hspace{1cm} 0=Q(1,1;-1,1)=(2\beta+2-\alpha)\beta \qquad \text{et} \qquad 0=Q(1,-1;1,1)=-(2\beta+2+\alpha)\beta,$$ {\it i.e.} si et seulement si $\alpha=0$\, et \,$\beta=-1$, auquel cas $\omega=\omega_{\hspace{0.3mm}7}=y^3\mathrm{d}x+x(3y^2-x^2)\mathrm{d}y.$
 \item [$\bullet$] Dans ce deuxième cas, il ne nous reste plus qu'à traiter l'éventualité $\mathcal{T}_\mathcal{H}=2\cdot\mathrm{T}_1+1\cdot\mathrm{T}_2$. Toujours d'après le Lemme~\ref{lem:2T1+1TR2}, $\omega$ est, à conjugaison près, de la forme
     \[
      \hspace{1cm}\omega=\left(x^3-3xy^2+\alpha\,y^3\right)\mathrm{d}x+\left(\delta\,x^3-3\delta\,xy^2+\beta\,y^3\right)\mathrm{d}y,\hspace{3mm} (\beta-\alpha\delta)\left((\beta-2)^2-(\alpha-2\delta)^2\right)\neq0\hspace{1mm};
     \]
     comme $\ItrH=y^2(y-x)(y+x)$ le $3$-tissu $\Leg\mathcal{H}$ est plat si et seulement si
     \[
     \left\{
      \begin{array}[l]{l}
      0\equiv\mathrm{d}\omega\Big|_{y=0}=3\delta\,x^{2}\mathrm{d}x\wedge\mathrm{d}y
      \\
      0=Q(1,1;-1,1)=(4+\beta-2\alpha-2\delta)(\beta-\alpha\delta)
      \\
      0=Q(1,-1;1,1)=(4+\beta+2\alpha+2\delta)(\beta-\alpha\delta),
     \end{array}
     \right.
    \]
     en vertu du Corollaire~\ref{cor:platitude-degre-3}. Il s'en suit que $\Leg\mathcal{H}$ est plat si et seulement si $\alpha=\delta=0$\, et \,$\beta=-4$, auquel cas $\omega=\omega_8=x(x^2-3y^2)\mathrm{d}x-4y^3\mathrm{d}y.$
\end{itemize}

\noindent\textsl{Troisième cas}: $\deg\mathcal{T}_{\mathcal{H}}=4.$ Pour examiner la platitude dans ce dernier cas, nous allons appliquer le Corollaire~\ref{cor:platitude-degre-3} aux différents modèles du Lemme~\ref{lem:deg-type=4}.
\begin{itemize}
  \item [$\bullet$] Si $\mathcal{T}_\mathcal{H}=3\cdot\mathrm{R}_1+1\cdot\mathrm{T}_1$, alors $\omega$ est du type
  \[
    \omega=y^2\left((2r+3)x-(r+2)y\right)\mathrm{d}x-x^2(x+ry)\mathrm{d}y
  \]
  avec $r(r+1)(r+2)(r+3)(2r+3)\neq0$. Nous avons $\ItrH=sx+ty$ où $s=2r+3$\, et \,$t=r(r+2)$; par suite le $3$-tissu $\Leg\mathcal{H}$ est plat si et seulement si $$0=Q(t,-s\hspace{0.2mm};s,t)=r(r+1)^2(r+2)^2(r+3)(2r+3)\left[r^2+3r+3\right],$$ {\it i.e.} si et seulement si $r=-\dfrac{3}{2}\pm\mathrm{ i}\dfrac{\sqrt{3}}{2}$. Dans les deux cas la $1$-forme $\omega$ est linéairement conjuguée à $$\omega_9=y^{2}\left((-3+\mathrm{i}\sqrt{3})x+2y\right)\mathrm{d}x+x^{2}\left((1+\mathrm{i}\sqrt{3})x-2\mathrm{i}\sqrt{3}y\right)\mathrm{d}y\hspace{1mm};$$ en effet si $r=-\dfrac{3}{2}-\mathrm{ i}\dfrac{\sqrt{3}}{2},$\, resp. $r=-\dfrac{3}{2}+\mathrm{ i}\dfrac{\sqrt{3}}{2}$,\, alors
  $$
  \hspace{1cm}\omega_9=-(1+\mathrm{i}\sqrt{3})\omega,
  \hspace{2cm}\text{resp}.\hspace{1.5mm}
  \omega_9=-2\varphi^*\omega,\quad \text{où}\hspace{1.5mm} \varphi(x,y)=(y,x).
  $$
  \item [$\bullet$] Si $\mathcal{T}_\mathcal{H}=2\cdot\mathrm{R}_1+2\cdot\mathrm{T}_1$, alors $\omega$ est de la forme
  \[
   \omega=s\hspace{0.12mm}y^2\left((2r+3)x-(r+2)y\right)\mathrm{d}x-x^2(x+ry)\mathrm{d}y
  \]
   avec $rs(s-1)(r+1)(r+2)(r+3)(2r+3)\left(s(2r+3)^2-r^2\right)\neq0$. Posons $t=2r+3$ et $u=r(r+2)$; nous avons $\ItrH=(y-x)(tx+uy).$ Donc $\Leg\mathcal{H}$ est plat si et seulement si
  \[
   \hspace{2.5cm}
   \left\{
   \begin{array}[l]{l}
   0=Q(1,1;-1,1)=-s(r+1)^2\left[s(r+2)+1\right]
   \\
   0=Q(u,-t;\hspace{0.2mm}t,u)=rs(r+1)^2(r+2)^2(2r+3)\left[s(2r+3)^2+(r+2)r^2\right],
   \end{array}
   \right.
  \]
  {\it i.e.} si et seulement si $r=\pm\sqrt{3}$\, et \,$s=-2+r$, car $rs(r+1)(r+2)(2r+3)\neq0.$ Dans les deux cas $\omega$ est linéairement conjuguée à $$\omega_{10}=(3x+\sqrt{3}y)y^2\mathrm{d}x+(3y-\sqrt{3}x)x^2\mathrm{d}y\hspace{1mm};$$ en effet si $(r,s)=(-\sqrt{3},-2-\sqrt{3})$,\, resp. $(r,s)=(\sqrt{3},-2+\sqrt{3})$, alors
  $$
  \hspace{1cm}\omega_{10}=\sqrt{3}\omega,
  \hspace{2cm}\text{resp}.\hspace{1.5mm}
  \omega_{10}=-\sqrt{3}\hspace{0.2mm}\varphi^*\omega,\quad \text{où}\hspace{1.5mm} \varphi(x,y)=(x,-y).
  $$
  \item [$\bullet$] Si $\mathcal{T}_\mathcal{H}=1\cdot\mathrm{R}_1+3\cdot\mathrm{T}_1$, alors $\omega$ est du type
  \[
   \omega=t\hspace{0.12mm}y^2\left((2r+3)x-(r+2)y\right)\mathrm{d}x-x^2(x+ry)\mathrm{d}(sy-x)
  \]
     avec\hspace{1mm} $rst(r+1)(r+2)(r+3)(2r+3)(s-t-1)(tu^3-r^2su-r^2v)\neq0,$\hspace{1mm} $u=2r+3$\, et \,$v=r(r+2).$ Puisque $\ItrH=y(y-x)(ux+vy)$ la courbure de $\Leg\mathcal{H}$ est holomorphe le long de $\mathcal{G}_{\mathcal{H}}(\{y(y-x)=0\})$ si et seulement si
  \[
   \hspace{2.5cm}
   \left\{
   \begin{array}[l]{l}
   0=Q(1,0\hspace{0.2mm};0,1)=st\left[(2r+3)s-(r+2)\right]
   \\
   0=Q(1,1;-\hspace{0.2mm}1,1)=-st(r+1)^2\left[(r+2)(t+1)+s\right],
   \end{array}
   \right.
 \]
  {\it i.e.} si et seulement si $s=\dfrac{r+2}{2r+3}$\, et \,$t=-\dfrac{2(r+2)}{2r+3},$ auquel cas $K(\Leg\mathcal{H})$ ne peut être holomorphe sur $\mathcal{G}_{\mathcal{H}}(\{ux+vy=0\})$ car $$\hspace{1cm}Q(v,-u\hspace{0.2mm};u,v)=12\,r(r+1)^3(r+2)^5(r+3)(2r+3)^{-2}\neq0.$$ Par conséquent la transformée de \textsc{Legendre} $\Leg\mathcal{H}$ de $\mathcal{H}$ ne peut être plate lorsque $\mathcal{T}_\mathcal{H}=1\cdot\mathrm{R}_1+3\cdot\mathrm{T}_1.$
  \item [$\bullet$] Si $\mathcal{T}_\mathcal{H}=4\cdot\mathrm{T}_1$, alors $\omega$ est de la forme
  \[
   \omega=uy^2\left((2r+3)x-(r+2)y\right)\mathrm{d}(y-sx)-x^2(x+ry)\mathrm{d}(ty-x),
  \]
   où $ur(r+1)(r+2)(r+3)(2r+3)(st-1)(su+t-u-1)(uv^4+suwv^3+r^2twv+r^2w^2)\neq0$, $v=2r+3$\, et \,$w=r(r+2).$ Comme $\ItrH=xy(y-x)(vx+wy)$ la courbure de $\Leg\mathcal{H}$ est holomorphe le long de $\mathcal{G}_{\mathcal{H}}(\{xy(y-x)=0\})$ si et seulement si
  \[
   \hspace{2.5cm}
   \left\{
   \begin{array}[l]{l}
   0=Q(0,-1;\hspace{0.2mm}1,0)=-u^2(r+2)^2(st-1)\left[rs+1\right]
   \\
   0=Q(1,0\hspace{0.2mm};0,1)=-u(st-1)\left[(2r+3)t-r-2\right]
   \\
   0=Q(1,1;\hspace{0.2mm}-1,1)=-u(r+1)^2(st-1)\left[(rs+2s+1)u-t-r-2\right],
   \end{array}
   \right.
  \]
   {\it i.e.} si et seulement si $s=-\dfrac{1}{r}$,\, $t=\dfrac{r+2}{2r+3}$\, et \,$u=-\dfrac{r(r+2)^2}{2r+3},$ auquel cas $$\hspace{-1cm}Q(w,-v\hspace{0.2mm};v,w)=16r(r+1)^5(r+2)^5(r+3)(2r+3)^{-2}\left[r^2+3r+3\right].$$ Par suite $\Leg\mathcal{H}$ est plat si et seulement si nous sommes dans l'un des deux cas suivants
   \begin{itemize}
   \item [(i)]  $r=-\dfrac{3}{2}+\mathrm{ i}\dfrac{\sqrt{3}}{2},\quad s=\dfrac{1}{2}+\mathrm{i}\dfrac{\sqrt{3}}{6},\quad
                 t=\dfrac{1}{2}-\mathrm{ i}\dfrac{\sqrt{3}}{6},\quad u=1\hspace{1mm};$
   \item [(ii)] $r=-\dfrac{3}{2}-\mathrm{ i}\dfrac{\sqrt{3}}{2},\quad s=\dfrac{1}{2}-\mathrm{i}\dfrac{\sqrt{3}}{6},\quad
                 t=\dfrac{1}{2}+\mathrm{ i}\dfrac{\sqrt{3}}{6},\quad u=1.$
   \end{itemize}
   Dans les deux cas la $1$-forme $\omega$ est linéairement conjuguée à $$\omega_{11}=(3x^3+3\sqrt{3}x^2y+3xy^2+\sqrt{3}y^3)\mathrm{d}x+(\sqrt{3}x^3+3x^2y+3\sqrt{3}xy^2+3y^3)\mathrm{d}y\hspace{1mm};$$ en effet dans les cas (i), resp. (ii) nous avons
   \begin{align*}
   &&\omega_{11}=3\varphi_1^*\omega,\quad \text{où}\hspace{1.5mm} \varphi_1=(x,\mathrm{e}^{-5\mathrm{i}\pi/6}\,y),
   \hspace{1cm}\text{resp}.\hspace{1.5mm}
   \omega_{11}=3\varphi_2^*\omega,\quad \text{où}\hspace{1.5mm} \varphi_2=(x,\mathrm{e}^{5\mathrm{i}\pi/6}\,y).
   \end{align*}
\end{itemize}
\end{proof}

\noindent Une particularité remarquable de la classification obtenue est que toutes les singularités des feuilletages $\mathcal{H}_{i}$, $i=1,\ldots,11$, sur la droite à l'infini sont non-dégénérées. Nous aurons besoin dans le prochain paragraphe des valeurs des indices $\mathrm{CS}(\mathcal{H}_{i},L_{\infty},s)$, $s\in\Sing\mathcal{H}_{i}\cap L_{\infty}$. Pour cela, nous avons calculé, pour chaque $i=1,\ldots,11$, le polynôme suivant (dit \textsl{polynôme de \textsc{Camacho-Sad} du feuilletage homogène} $\mathcal{H}_{i}$)
\begin{align*}
\mathrm{CS}_{\mathcal{H}_{i}}(\lambda)=\prod\limits_{s\in\Sing\mathcal{H}_{i}\cap L_{\infty}}(\lambda-\mathrm{CS}(\mathcal{H}_{i},L_{\infty},s)).
\end{align*}
\noindent Le tableau suivant résume les types et les polynômes de \textsc{Camacho-Sad} des feuilletages $\mathcal{H}_{i}$, $i=1,\ldots,11$.
\begingroup
\renewcommand*{\arraystretch}{1.5}
\begin{table}[h]
\begin{center}
\begin{tabular}{|c|c|c|}\hline
$i$  &  $\mathcal{T}_{\mathcal{H}_{\hspace{0.2mm}i}}$              &  $\mathrm{CS}_{\mathcal{H}_{\hspace{0.2mm}i}}(\lambda)$ \\\hline
$1$  &  $2\cdot\mathrm{R}_2$                                       &  $(\lambda-1)^{2}(\lambda+\frac{1}{2})^{2}$              \\\hline
$2$  &  $2\cdot\mathrm{T}_2$                                       &  $(\lambda-\frac{1}{4})^{4}$                              \\\hline
$3$  &  $2\cdot\mathrm{R}_1+1\cdot\mathrm{R}_2$                    &  $(\lambda-1)^{3}(\lambda+2)$                              \\\hline
$4$  &  $2\cdot\mathrm{R}_1+1\cdot\mathrm{T}_2$                    &  $(\lambda-1)^{2}(\lambda+\frac{1}{2})^{2}$                 \\\hline
$5$  &  $1\cdot\mathrm{R}_1+1\cdot\mathrm{T}_1+1\cdot\mathrm{R}_2$ &  $(\lambda-1)^{2}(\lambda+\frac{1}{5})(\lambda+\frac{4}{5})$ \\\hline
$6$  &  $1\cdot\mathrm{R}_1+1\cdot\mathrm{T}_1+1\cdot\mathrm{T}_2$ &  $(\lambda-1)(\lambda+\frac{2}{7})(\lambda-\frac{1}{7})^{2}$  \\\hline
$7$  &  $2\cdot\mathrm{T}_1+1\cdot\mathrm{R}_2$                    &  $(\lambda-1)(\lambda-\frac{1}{4})(\lambda+\frac{1}{8})^{2}$   \\\hline
$8$  &  $2\cdot\mathrm{T}_1+1\cdot\mathrm{T}_2$                    &  $(\lambda-\frac{1}{10})^{2}(\lambda-\frac{2}{5})^{2}$          \\\hline
$9$  &  $3\cdot\mathrm{R}_1+1\cdot\mathrm{T}_1$                    &  $(\lambda-1)^{3}(\lambda+2)$                                    \\\hline
$10$ &  $2\cdot\mathrm{R}_1+2\cdot\mathrm{T}_1$                    &  $(\lambda-1)^{2}(\lambda+\frac{1}{2})^{2}$                       \\\hline
$11$ &  $4\cdot\mathrm{T}_1$                                       &  $(\lambda-\frac{1}{4})^{4}$                                       \\\hline
\end{tabular}
\end{center}
\bigskip
\caption{Types et polynômes de \textsc{Camacho-Sad} des feuilletages homogènes donnés par le Théorème~\ref{thm:Class-Homog3-Plat}}\label{tab:CS(lambda)}
\end{table}
\endgroup

\newpage

\section{Feuilletages à singularités non-dégénérées et de transformée de \textsc{Legendre} plate}\label{sec:Fermat}
\bigskip

\noindent L'ensemble $\mathbf{F}(d)$ des feuilletages de degré $d$ sur $\pp$ est un ouvert de \textsc{Zariski} dans l'espace projectif $\mathbb{P}^{(d+2)^{2}-2}$. Le groupe des automorphismes de $\pp$ agit sur $\mathbf{F}(d)$; l'orbite d'un élément $\F\in\mathbf{F}(d)$ sous l'action de $\mathrm{Aut}(\pp)=\mathrm{PGL}_3(\mathbb{C})$ est notée $\mathcal{O}(\F)$, \emph{voir} \cite{CDGBM10}.  Le sous-ensemble $\mathbf{FP}(d)$ de $\mathbf{F}(d)$ formé des $\F\in\mathbf{F}(d)$ tels que $\Leg\F$ soit plat est un fermé de \textsc{Zariski} de $\mathbf{F}(d)$. Signalons aussi que si $\F\in\mathbf{FP}(d)$ alors l'adhérence $\overline{\mathcal{O}(\F)}$ (dans $\mathbf{F}(d)$) de $\mathcal{O}(\F)$ est contenue dans $\mathbf{FP}(d)$.

\noindent  Parmi les éléments de $\mathbf{FP}(d)$ n'ayant que des singularités non-dégénérées, il y a le \textsl{feuilletage de \textsc{Fermat}} $\F^{d}$ de degré $d$ défini en carte affine par la $1$-forme
$$\omega_{F}^{d}=(x^{d}-x)\mathrm{d}y-(y^{d}-y)\mathrm{d}x\hspace{1mm};$$
en effet, d'une part $\Leg\F^{d}$ est plat car il est algébrisable d'après \cite[Proposition~5.2]{MP13}; d'autre part, un calcul élémentaire montre que toutes les singularités du feuilletage $\F^{d}$ sont non-dégénérées. Nous savons aussi d'après \cite[Théorème~3]{MP13} que $\overline{\mathcal{O}(\F^{d})}$ est une composante irréductible de $\mathbf{FP}(d)$ pour $d\neq 4$.
\smallskip

\noindent Le théorème suivant est le résultat principal de ce paragraphe.

\begin{thm}\label{thm:Fermat}
Soit $\F$ un feuilletage de degré $3$ sur $\pp$. Supposons que toutes ses singularités soient non-dégénérées et que son $3$-tissu dual $\Leg\F$ soit plat. Alors $\F$ est linéairement conjugué au feuilletage de \textsc{Fermat} $\F^{3}$ défini par la $1$-forme $\omega_{F}^{3}=(x^{3}-x)\mathrm{d}y-(y^{3}-y)\mathrm{d}x.$
\end{thm}

\begin{rem}
L'ensemble $\mathbf{FP}(4)$ contient des feuilletages à singularités non-dégénérées et qui ne sont pas conjugués au feuilletage $\F^{4},$ {\it e.g.} la famille $(\F_{\lambda}^{4})_{\lambda\in\C}$ de feuilletages définis par
$$\omega_{F}^{4}+\lambda((x^{3}-1)y^{2}\mathrm{d}y-(y^{3}-1)x^{2}\mathrm{d}x).$$
En effet, d'après \cite[Théorème 8.1]{MP13}, pour tout $\lambda$ fixé dans $\mathbb{C}$, $\F_{\lambda}^{4}\in\mathbf{FP}(4)$; de plus un calcul facile montre que $\F_{\lambda}^{4}$ est à singularités non-dégénérées. Mais, si $\lambda$ est non nul alors $\F^{4}_{\lambda}$ n'est pas conjugué à $\F^{4}$ car il n'est pas convexe.
\end{rem}

\noindent La démonstration du Théorème~\ref{thm:Fermat} repose sur le Théorème~\ref{thm:Class-Homog3-Plat} de classification des feuilletages homogènes appartenant à $\mathbf{FP}(3)$, et sur les trois résultats qui suivent, dont les deux premiers sont valables en degré quelconque.
\smallskip

\noindent Notons d'abord que le feuilletage $\F^{d}$ possède trois singularités radiales d'ordre maximal $d-1$, non alignées. La proposition suivante montre que cette propriété caractérise l'orbite $\mathcal{O}(\F^{d})$.

\begin{pro}\label{pro:Fermat-d}
Soit $\F$ un feuilletage de degré $d$ sur $\pp$ ayant trois singularités radiales d'ordre maximal $d-1$, non alignées. Alors $\F$ est linéairement conjugué au feuilletage de \textsc{Fermat} $\F^{d}.$
\end{pro}

\begin{proof}
Par hypothèse $\F$ possède trois points singuliers $m_j,j=1,2,3,$ non alignés vérifiant $\nu(\F,m_j)=1$ et $\tau(\F,m_j)=d$. D'après \cite[Proposition~2, page~23]{Bru00}, les égalités $\tau(\F,m_j)=\tau(\F,m_l)=d$ avec $l\neq j$ impliquent que la droite $(m_jm_l)$ est invariante par $\F$. Choisissons des coordonnées homogènes $[x:y:z]\in\pp$ telles que $m_1=[0:0:1],\,m_2=[0:1:0]$\, et \,$m_3=[1:0:0]$. Les égalités $\nu(\F,m_1)=1$ et $\tau(\F,m_1)=d$, combinées avec le fait que $(m_2m_3)=(z=0)$ est $\F$-invariante, assurent que toute $1$-forme $\omega$ décrivant $\F$ dans la carte affine $z=1$ est du type
\begin{align*}
\omega=(x\mathrm{d}y-y\mathrm{d}x)(\gamma+C_{1}(x,y)+\cdots +C_{d-2}(x,y))+A_{d}(x,y)\mathrm{d}x+B_{d}(x,y)\mathrm{d}y
\end{align*}
avec $\gamma\neq 0,\hspace{3mm}A_{d},B_{d}\in\mathbb{C}[x,y]_d,\hspace{3mm}C_{k}\in\mathbb{C}[x,y]_k$ \hspace{1mm}pour $k=1,\ldots,d-2.$
\vspace{1mm}

\noindent Dans la carte affine $y=1$ le feuilletage $\F$ est donné par
\begin{align*}
\theta=-(\gamma\hspace{0.1mm}z^{d}+C_{1}(x,1)z^{d-1}+\cdots +C_{d-2}(x,1)z^{2})\mathrm{d}x+A_{d}(x,1)(z\mathrm{d}x-x\mathrm{d}z)-B_{d}(x,1)\mathrm{d}z\hspace{1mm};
\end{align*}
nous avons $\theta\wedge(z\mathrm{d}x-x\mathrm{d}z)=z\hspace{0.1mm}Q(x,z)\mathrm{d}x\wedge\mathrm{d}z$, avec
$$Q(x,z)=x\left[\gamma\hspace{0.1mm}z^{d-1}+C_{1}(x,1)z^{d-2}+\cdots+C_{d-2}(x,1)z\right]+B_{d}(x,1).$$
L'égalité $\tau(\F,m_2)=d$ entraîne alors que le polynôme $Q\in\mathbb{C}[x,z]$ est homogène de degré $d$, ce qui permet d'écrire $B_d(x,y)=\beta\hspace{0.1mm}x^d$\, et \,$C_k(x,y)=\delta_{k}x^k,$\, $\beta,\delta_k\in\mathbb{C}.$ Par suite nous avons $J^{1}_{(0,0)}\theta=A_{d}(0,1)(z\mathrm{d}x-x\mathrm{d}z)$; alors l'égalité $\nu(\F,m_2)=1$ assure que $A_{d}(0,1)\neq0.$
\smallskip

\noindent De la même manière, en se plaçant dans la carte affine $x=1$ et en écrivant explicitement les égalités $\tau(\F,m_3)=d$\, et \,$\nu(\F,m_3)=1$, nous obtenons que $B_{d}(1,0)\neq0,$\, $A_d(x,y)=\alpha\hspace{0.1mm}y^d$\, et \,$C_k(x,y)=\varepsilon_{k}y^k,$\, $\alpha,\varepsilon_k\in\mathbb{C}.$ Donc $\alpha\beta\neq0,$\, les $C_k$ sont tous nuls et \,$\omega$\, est du type $$\omega=\gamma(x\mathrm{d}y-y\mathrm{d}x)+\alpha y^{d}\mathrm{d}x+\beta x^{d}\mathrm{d}y.$$
Écrivons $\alpha=\gamma\mu^{1-d}$\, et \,$\beta=-\gamma\hspace{0.1mm}\lambda^{\hspace{-0.1mm}1-d}.$ Quitte à remplacer $\omega$ par $\varphi^*\omega,$ où $\varphi(x,y)=(\lambda\hspace{0.1mm}x,\mu\hspace{0.1mm}y),$ le feuilletage $\F$ est défini, dans les coordonnées affines $(x,y),$ par la $1$-forme
$$\omega_{F}^{d}=(x^{d}-x)\mathrm{d}y-(y^{d}-y)\mathrm{d}x.$$
\end{proof}

\noindent La proposition suivante permet de ramener l'étude de la platitude au cadre homogène:

\begin{pro}\label{pro:F-dégénère-H}
Soit $\F$ un feuilletage de degré $d\geq1$ sur $\pp$ ayant une droite invariante $L.$ Supposons que toutes les singularités de $\F$ sur $L$ soient non-dégénérées. Il existe un feuilletage homogène $\mathcal{H}$ de degré $d$ sur $\pp$ ayant les propriétés suivantes
\vspace{1mm}
\begin{itemize}
\item [$\bullet$] $\mathcal{H}\in\overline{\mathcal{O}(\F)}$;
\vspace{0.5mm}
\item [$\bullet$] $L$ est invariante par $\mathcal{H}$;
\vspace{0.5mm}
\item [$\bullet$] $\Sing\mathcal{H}\cap L=\Sing\F\cap L$;
\vspace{0.5mm}
\item [$\bullet$] $\forall\hspace{1mm}s\in\Sing\mathcal{H}\cap L,\hspace{1mm}
                   \mu(\mathcal{H},s)=1
                   \hspace{2mm}\text{et}\hspace{2mm}
                   \mathrm{CS}(\mathcal{H},L,s)=\mathrm{CS}(\F,L,s)$.
\end{itemize}
\vspace{1mm}
Si de plus $\Leg\F$ est plat, alors $\Leg\mathcal{H}$ l'est aussi.
\end{pro}

\begin{proof}
Choisissons des coordonnées homogènes $[x:y:z]\in\pp$ telles que $L=(z=0)$; comme $L$ est $\F$-invariante, $\F$ est défini dans la carte affine $z=1$ par une $1$-forme $\omega$ du type
$$\omega=\sum_{i=0}^{d}(A_i(x,y)\mathrm{d}x+B_i(x,y)\mathrm{d}y),$$
où les $A_i,\,B_i$ sont des polynômes homogènes de degré $i$.

\noindent Montrons par l'absurde que $\mathrm{pgcd}(A_d,B_d)=1$; supposons donc que $\mathrm{pgcd}(A_d,B_d)\neq1.$ Quitte à conjuguer $\omega$ par une transformation linéaire de $\mathbb{C}^2=(z=1)$, nous pouvons nous ramener à
$$A_{d}(x,y)=x\hspace{0.3mm}\widetilde{A}_{d-1}(x,y) \qquad\text{et}\qquad B_{d}(x,y)=x\widetilde{B}_{d-1}(x,y)$$
pour certains $\widetilde{A}_{d-1},\,\widetilde{B}_{d-1}$ dans $\mathbb{C}[x,y]_{d-1}$; alors $s_{0}=[0:1:0]\in L$ est un point singulier de $\F$. Dans la carte affine $y=1$, le feuilletage $\F$ est donné par
\begin{eqnarray*}
\hspace{3.6cm}\theta\hspace{-1mm}&=&\hspace{-1mm}\sum_{i=0}^{d}z^{d-i}[A_i(x,1)(z\mathrm{d}x-x\mathrm{d}z)-B_i(x,1)\mathrm{d}z]
\\
\hspace{3.6cm}\hspace{-1mm}&=&\hspace{-1mm}[A_{d}(x,1)+z\hspace{0.3mm}A_{d-1}(x,1)+\cdots](z\mathrm{d}x-x\mathrm{d}z)-
                [B_{d}(x,1)+zB_{d-1}(x,1)+\cdots]\mathrm{d}z.
\end{eqnarray*}
Le $1$-jet de $\theta$ au point singulier $s_{0}=(0,0)$ s'écrit $-[\widetilde{B}_{d-1}(0,1)x+B_{d-1}(0,1)z]\mathrm{d}z$; ce qui implique que $\mu(\F,s_{0})>1$: contradiction avec l'hypothèse que toute singularité de $\F$ située sur $L$ est non-dégénérée.

\noindent Il s'en suit que la $1$-forme $\omega_d=A_d(x,y)\mathrm{d}x+B_d(x,y)\mathrm{d}y$ définit bien un feuilletage homogène de degré $d$ sur $\pp$, que nous notons $\mathcal{H}$. Il est évident que $L$ est $\mathcal{H}$-invariante et que  $\Sing\mathcal{F}\cap L=\Sing\mathcal{H}\cap L$. Considérons la famille d'homothéties $\varphi=\varphi_{\varepsilon}=(\frac{x}{\varepsilon},\frac{y}{\varepsilon}).$ Nous avons $$\varepsilon^{d+1}\varphi^*\omega=\sum_{i=0}^{d}(\varepsilon^{d-i}A_i(x,y)\mathrm{d}x+\varepsilon^{d-i}B_i(x,y)\mathrm{d}y)$$
qui tend vers $\omega_d$ lorsque $\varepsilon$ tend vers $0$; il en résulte que $\mathcal{H}\in\overline{\mathcal{O}(\F)}.$
\smallskip

\noindent Montrons que $\mathcal{H}$ vérifie la quatrième propriété de l'énoncé. Soit $s\in\Sing\mathcal{H}\cap L$. Quitte à conjuguer $\omega$ par un isomorphisme linéaire de $\mathbb{C}^2=(z=0)$, nous pouvons supposer que $s=[0:1:0]$; il existe donc un polynôme $\widehat{B}_{d-1}\in\mathbb{C}[x,y]_{d-1}$ tel que $B_{d}(x,y)=x\widehat{B}_{d-1}(x,y)$. Le feuilletage $\mathcal{H}$ est décrit dans la carte affine $y=1$ par $$\theta_d=A_d(x,1)(z\mathrm{d}x-x\mathrm{d}z)-B_d(x,1)\mathrm{d}z.$$
\noindent Posons $\lambda=A_{d}(0,1)$\, et \,$\nu=A_{d}(0,1)+\widehat{B}_{d-1}(0,1)$. Le $1$-jet de $\theta_d$ en $s=(0,0)$ s'écrit $J^{1}_{(0,0)}\theta_{d}=\lambda\hspace{0.1mm}z\mathrm{d}x-\nu\hspace{0.1mm}x\mathrm{d}z$, et celui de $\theta$ est donné par $J^{1}_{(0,0)}\theta=\lambda\hspace{0.1mm}z\mathrm{d}x-\nu\hspace{0.1mm}x\mathrm{d}z-z\hspace{0.1mm}B_{d-1}(0,1)\mathrm{d}z$. L'hypothèse $\mu(\F,s)=1$ signifie que $\lambda\nu$ est non nul. Par suite $\mu(\mathcal{H},s)=1$\, et \,$\mathrm{CS}(\mathcal{H},L,s)=\mathrm{CS}(\F,L,s)=\frac{\lambda}{\nu}$.
\smallskip

\noindent L'implication\, $K(\Leg\F)\equiv0\hspace{0.3mm}\Longrightarrow\hspace{0.3mm}K(\Leg\mathcal{H})\equiv0$\, découle du fait que $\mathcal{H}\in\overline{\mathcal{O}(\F)}.$
\end{proof}

\noindent Nous illustrons le résultat précédent en l'appliquant au feuilletage $\F^{d}$.

\begin{eg}
Le feuilletage de \textsc{Fermat} $\F^{d}$ est donné en coordonnées homogènes par la $1$-forme
\[x^{d}(y\mathrm{d}z-z\mathrm{d}y)+y^{d}(z\mathrm{d}x-x\mathrm{d}z)+z^{d}(x\mathrm{d}y-y\mathrm{d}x).\] Il possède les $3d$ droites invariantes suivantes :
\begin{enumerate}
\item [(a)] $x=0$,\, $y=0$,\, $z=0$;
\item [(b)] $y=\zeta x$,\, $y=\zeta z$,\, $x=\zeta z$\,  avec $\zeta^{d-1}=1$.
\end{enumerate}
Les droites de la famille (a) (resp. (b)) donnent lieu à $3$ (resp. $3d-3$) feuilletages homogènes appartenant à $\overline{\mathcal{O}(\F^{d})}\subset\mathbf{FP}(d)$ et de type $2\cdot\mathrm{R}_{d-1}$ (resp. $1\cdot\mathrm{R}_{d-1} + (d-1)\cdot\mathrm{R}_{1}$). Ceux qui sont de type $2\cdot\mathrm{R}_{d-1}$ sont tous conjugués à $\mathcal{H}^{d}_{1}$, d'après la Proposition~\ref{pro:omega1-omega2}, et ceux qui sont de type $1\cdot\mathrm{R}_{d-1}+(d-1)\cdot\mathrm{R}_{1}$ sont tous conjugués au feuilletage défini par
$$
(y^{d-1}-dx^{d-1})y\mathrm{d}x+(d-1)x^{d}\mathrm{d}y.
$$
Pour $d=3$ ce dernier feuilletage est conjugué au feuilletage $\mathcal{H}_{3}^{d,1}$ donné par la Proposition~\ref{pro:omega3-omega4}, mais ce n'est plus le cas pour $d\ge4$.
\end{eg}

\begin{rem}\label{rem:droite-inva}
Si $\F$ est un feuilletage de degré $d$ sur $\pp$ et si $m$ est un point singulier de $\F$, nous avons l'encadrement $\sigma(\F,m)\leq\tau(\F,m)+1\leq d+1,$ où $\sigma(\F,m)$ désigne le nombre de droites (distinctes) invariantes par $\F$ et qui passent par $m.$
\end{rem}

\noindent Le lemme technique suivant joue un rôle clé dans la démonstration du Théorème~\ref{thm:Fermat}.

\begin{lem}\label{lem:2-droite-invar}
Soit $\F$ un feuilletage de degré $3$ sur $\pp.$ Si le $3$-tissu $\Leg\F$ est plat et si $\F$ possède une singularité $m$ non-dégénérée vérifiant $\mathrm{BB}(\F,m)\not\in\{4,\frac{16}{3}\}$, alors par le point $m$ il passe exactement deux droites invariantes par $\F$, {\it i.e.} $\sigma(\F,m)=2.$
\end{lem}

\begin{proof} Les deux conditions $\mu(\F,m)=1$ et $\mathrm{BB}(\F,m)\neq4$ assurent l'existence d'une carte affine $(x,y)$ de $\pp$ dans laquelle
$m=(0,0)$ et $\F$ est défini par une $1$-forme du type $\theta_1+\theta_2+\theta_3+\theta_4,$ où
\begin{align*}
\hspace{6mm}&  \theta_1=\lambda\hspace{0.1mm}y\mathrm{d}x+\mu\hspace{0.1mm}x\mathrm{d}y,&&
               \theta_2=\left(\sum_{i=0}^{2}\alpha_{i}\hspace{0.1mm}x^{2-i}y^{i}\right)\mathrm{d}x+
                        \left(\sum_{i=0}^{2}\beta_{i}\hspace{0.1mm}x^{2-i}y^{i}\right)\mathrm{d}y,\\
\hspace{6mm}&  \theta_3=\left(\sum_{i=0}^{3}a_{i}\hspace{0.1mm}x^{3-i}y^{i}\right)\mathrm{d}x+
                        \left(\sum_{i=0}^{3}b_{i}\hspace{0.1mm}x^{3-i}y^{i}\right)\mathrm{d}y,&&
               \theta_4=\left(\sum_{i=0}^{3}c_{i}\hspace{0.1mm}x^{3-i}y^{i}\right)(x\mathrm{d}y-y\mathrm{d}x),
\end{align*}
avec $\lambda\mu(\lambda+\mu)\neq0$; comme $\mathrm{BB}(\F,m)\neq\frac{16}{3}$ nous avons $\lambda\mu(\lambda+\mu)(\lambda+3\mu)(3\lambda+\mu)\neq0.$
\medskip

\noindent Commençons par montrer que $\alpha_0=0$. Supposons par l'absurde que $\alpha_0\neq0$. Soit $(p,q)$ la carte affine de $\pd$ associée à la droite $\{px-qy=1\}\subset{\mathbb{P}^{2}_{\mathbb{C}}}$; le $3$-tissu $\mathrm{Leg}\mathcal{F}$ est donné par la $3$-forme symétrique
\begin{align*}
&\check{\omega}=\left[\left(\beta_2\hspace{0.1mm}p+\alpha_2\hspace{0.1mm}q-\lambda\hspace{0.1mm}q^2\right)\mathrm{d}p^{2}+
                \left(\beta_1\hspace{0.1mm}p+\alpha_1\hspace{0.1mm}q+\lambda\hspace{0.1mm}pq-\mu\hspace{0.1mm}pq\right)\mathrm{d}p\mathrm{d}q+
                \left(\beta_0\hspace{0.1mm}p+\alpha_0\hspace{0.1mm}q+\mu\hspace{0.1mm}p^2\right)\mathrm{d}q^{2}\right](p\mathrm{d}q-q\mathrm{d}p)\\
&\hspace{0.7cm}+q\left(a_3\mathrm{d}p^{3}+a_2\mathrm{d}p^{2}\mathrm{d}q+a_1\mathrm{d}p\mathrm{d}q^{2}+a_0\mathrm{d}q^{3}\right)+
               p\left(b_3\mathrm{d}p^{3}+b_2\mathrm{d}p^{2}\mathrm{d}q+b_1\mathrm{d}p\mathrm{d}q^{2}+b_0\mathrm{d}q^{3}\right)\\
&\hspace{0.7cm}+c_3\mathrm{d}p^{3}+c_2\mathrm{d}p^{2}\mathrm{d}q+c_1\mathrm{d}p\mathrm{d}q^{2}+c_0\mathrm{d}q^{3}.
\end{align*}
Considérons la famille d'automorphismes $\varphi=\varphi_{\varepsilon}=(\alpha_0\hspace{0.1mm}\varepsilon^{-1}\hspace{0.1mm}p,\hspace{0.1mm}\alpha_0\hspace{0.1mm}\varepsilon^{-2}\hspace{0.1mm}q).$ Nous constatons que
\begin{align*}
&\check{\omega}_{0}:
=\lim_{\varepsilon\to 0}\varepsilon^9\alpha_{0}^{-6}\varphi^*\check{\omega}
=(p\mathrm{d}q-q\mathrm{d}p)\left(-\lambda\hspace{0.1mm}q^2\mathrm{d}p^2+pq(\lambda-\mu)\mathrm{d}p\mathrm{d}q+(\mu\hspace{0.1mm}p^2+q)\mathrm{d}q^{2}\right).
\end{align*}

\noindent Puisque $\mu$ est non nul $\check{\omega}_{0}$ définit un $3$-tissu $\mathcal{W}_{0}$, qui appartient évidemment à $\overline{\mathcal{O}(\Leg\F)}$. L'image réciproque de $\W_0$ par l'application rationnelle $\psi(p,q)=\left(\lambda(p+q),-\lambda(\lambda+\mu)^2pq\right)$ s'écrit $\psi^*\mathcal{W}_{0}=\mathcal{F}_{1}\boxtimes\mathcal{F}_{2}\boxtimes\mathcal{F}_{3}$, où
\begin{align*}
&\F_1\hspace{0.1mm}:\hspace{0.1mm}q^2\mathrm{d}p+p^2\mathrm{d}q=0,&&
\F_2\hspace{0.1mm}:\hspace{0.1mm}\mu q^2\mathrm{d}p+p(\lambda q+\mu q-\lambda p)\mathrm{d}q=0,&&
\F_3\hspace{0.1mm}:\hspace{0.1mm}\mu p^2\mathrm{d}q+q(\lambda p+\mu p-\lambda q)\mathrm{d}p=0.
\end{align*}
\noindent Un calcul direct, utilisant la formule (\ref{equa:eta-rst}), conduit à
\begin{small}
\begin{align*}
\eta(\psi^{*}\mathcal{W}_0)=
\frac{5(\lambda+\mu)p^2-(8\lambda+7\mu)pq+(3\lambda+4\mu)q^2}{(\lambda+\mu)p(p-q)^2}\mathrm{d}p+
\frac{5(\lambda+\mu)q^2-(8\lambda+7\mu)pq+(3\lambda+4\mu)p^2}{(\lambda+\mu)q(p-q)^2}\mathrm{d}q
\end{align*}
\end{small}
\noindent de sorte que
$$
K(\psi^{*}\mathcal{W}_0)=\mathrm{d}\eta(\psi^{*}\mathcal{W}_0)
=-\frac{4\mu(p+q)}{(\lambda+\mu)(p-q)^3}\mathrm{d}p\wedge\mathrm{d}q\not\equiv0\hspace{1mm};
$$
\noindent comme $\Leg\F$ est plat par hypothèse, il en est de même pour $\W_0$; par suite $K(\psi^{*}\mathcal{W}_0)=\psi^{*}K(\mathcal{W}_0)=0,$ ce qui est absurde. D'où l'égalité $\alpha_0=0.$
\medskip

\noindent Montrons maintenant que $a_0=0$. Raisonnons encore par l'absurde en supposant $a_0\neq0$. Le feuilletage $\F$ est décrit dans la carte affine $(x,y)$ par $\theta=\theta_1+\theta_2+\theta_3+\theta_4$\hspace{0.1mm} avec $\alpha_0=0.$ En faisant agir la transformation linéaire diagonale  $(\varepsilon\hspace{0,1mm}x\hspace{0,1mm},\hspace{0,1mm}a_0\hspace{0,1mm}\varepsilon^{3}y)$ sur $\theta$ puis en passant à la limite lorsque $\varepsilon\to0$ nous obtenons
$$\theta_0=\lambda\hspace{0.1mm}y\mathrm{d}x+\mu\hspace{0.1mm}x\mathrm{d}y+x^3\mathrm{d}x$$
qui définit un feuilletage de degré trois $\F_0\in\overline{O(\F)}.$ Notons $\mathrm{I}_{0}=\mathrm{I}_{\mathcal{F}_{0}}^{\hspace{0.2mm}\mathrm{tr}}$, $\mathcal{G}_{0}=\mathcal{G}_{\mathcal{F}_0}$ et $\mathrm{I}_{0}^{\perp}=\overline{\mathcal{G}_{0}^{-1}(\mathcal{G}_{0}(\mathrm{I}_{0}))\setminus\mathrm{I}_{0}},$ où l'adhérence est prise au sens ordinaire. Un calcul élémentaire montre que
\begin{align*}
&\mathcal{G}_{0}(x,y)=\left(\dfrac{x^3+\lambda\hspace{0.1mm}y}{x(x^3+\lambda\hspace{0.1mm}y+\mu\hspace{0.1mm}y)},
-\dfrac{\mu}{x^3+\lambda\hspace{0.1mm}y+\mu\hspace{0.1mm}y}\right),&\quad
\mathrm{I}_{0}=\{(x,y)\in\mathbb{C}^2\hspace{1mm}\colon(\lambda-2\mu)x^3+\lambda(\lambda+\mu)y=0\}\subset\pp
\end{align*}
et que la courbe $\mathrm{I}_{0}^{\perp}$ a pour équation affine $f(x,y)=y-\nu\hspace{0.1mm}x^3=0$, où $\nu=-\dfrac{4\lambda+\mu}{4\lambda(\lambda+\mu)}.$ Comme $\Leg\F$ est plat, $\Leg\F_0$ l'est aussi. Or, d'après \cite[Corollaire~4.6]{BFM13}, le $3$-tissu $\Leg\F_0$ est plat si et seulement si $\mathrm{I}_{0}^{\perp}$ est invariante par $\F_0,$ {\it i.e.} si et seulement si
$$0\equiv\mathrm{d}f\wedge\theta_0\Big|_{y=\nu x^3}=3(3\lambda+\mu)\mu\hspace{0.2mm}x^3\mathrm{d}x\wedge\mathrm{d}y\hspace{1mm};$$
d'où $\mu(3\lambda+\mu)=0$: contradiction. Donc $a_0=\alpha_0=0$, ce qui signifie que la droite $(y=0)$ est $\F$-invariante.
\medskip

\noindent Ce qui précède montre également que l'invariance de la droite $(y=0)$ par $\F$ découle uniquement du fait que $\lambda\mu(\lambda+\mu)(3\lambda+\mu)\neq0$ et de l'hypothèse que $\Leg\F$ est plat. Ainsi en permutant les coordonnées $x$ et $y$, la condition $\lambda\mu(\lambda+\mu)(\lambda+3\mu)\neq0$ permet de déduire que $\beta_2=b_3=0$, {\it i.e.} que la droite $(x=0)$ est aussi invariante par $\F$.
\smallskip

\noindent La singularité $m$ de $\F$ n'est pas radiale car $\mathrm{BB}(\F,m)\neq4$; de plus $\nu(\F,m)=1$ car $\mu(\F,m)=1$. Il s'en suit que $\tau(\F,m)=1$; d'après la Remarque \ref{rem:droite-inva}, nous avons $\sigma(\F,m)\leq\tau(\F,m)+1=2,$ d'où l'énoncé.
\end{proof}

\noindent Avant de commencer la démonstration du Théorème~\ref{thm:Fermat}, rappelons (\emph{voir} \cite{Bru00}) que si $\F$ est un feuilletage de degré $d$ sur $\pp$ alors
\begin{align}\label{equa:Darboux-BB}
&\sum_{s\in\mathrm{Sing}\F}\mu(\F,s)=d^2+d+1
&&\text{et}&&
\sum_{s\in\mathrm{Sing}\F}\mathrm{BB}(\F,s)=(d+2)^2.
\end{align}

\begin{proof}[\sl D\'emonstration du Théorème~\ref{thm:Fermat}]
\'{E}crivons $\Sing\F=\Sigma^{0}\cup\Sigma^{1}\cup\Sigma^{2}$ avec
\begin{align*}
&\Sigma^{0}=\{s\in\Sing\F\hspace{1mm}\colon \mathrm{BB}(\F,s)=\tfrac{16}{3}\},&&
\Sigma^{1}=\{s\in\Sing\F\hspace{1mm}\colon \mathrm{BB}(\F,s)=4\},&&
\Sigma^{2}=\Sing\F\setminus(\Sigma^{0}\cup\Sigma^{1})
\end{align*}
et notons $\kappa_i=\#\hspace{0.5mm}\Sigma^{i},\,i=0,1,2.$ Par hypothèse, $\F$ est de degré $3$ et toutes ses singularités ont leur nombre de \textsc{Milnor} $1$. Les formules (\ref{equa:Darboux-BB}) impliquent alors que
\begin{align}\label{equa:Dar-BB-3}
&\#\hspace{0.5mm}\Sing\F=\kappa_0+\kappa_1+\kappa_2=13 &&\text{et} &&\tfrac{16}{3}\kappa_0+4\kappa_1+\sum_{s\in\Sigma^{2}}\mathrm{BB}(\F,s)=25\hspace{1mm};
\end{align}
il en résulte que $\Sigma^{2}$ est non vide. Soit $m$ un point de $\Sigma^{2}$; d'après le Lemme~\ref{lem:2-droite-invar} il passe par $m$ exactement deux droites $\ell_{m}^{(1)}$ et $\ell_{m}^{(2)}$ invariantes par $\F$. Alors, pour $i=1,2$, la Proposition~\ref{pro:F-dégénère-H} assure l'existence d'un feuilletage homogène $\mathcal{H}^{(i)}_{m}$ de degré $3$ sur $\pp$ appartenant à $\overline{\mathcal{O}(\F)}$ et tel que la droite $\ell_{m}^{(i)}$ soit $\mathcal{H}^{(i)}_{m}$-invariante. Comme $\Leg\F$ est plat par hypothèse, il en est de même pour $\Leg\mathcal{H}^{(1)}_{m}$ et $\Leg\mathcal{H}^{(2)}_{m}$. Donc chacun des $\mathcal{H}^{(i)}_{m}$ est linéairement conjugué à l'un des onze feuilletages homogènes donnés par le Théorème~\ref{thm:Class-Homog3-Plat}. Pour $i=1,2,$ la Proposition~\ref{pro:F-dégénère-H} assure aussi que
\vspace{1mm}
\begin{itemize}
\item [($\mathfrak{a}$)] $\Sing\F\cap\ell_{m}^{(i)}=\Sing\mathcal{H}^{(i)}_{m}\cap\ell_{m}^{(i)}$;
\vspace{0.5mm}
\item [($\mathfrak{b}$)] $\forall\hspace{1mm}s\in\Sing\mathcal{H}^{(i)}_{m}\cap\ell_{m}^{(i)},\quad
                   \mu(\mathcal{H}^{(i)}_{m},s)=1\hspace{2mm}\text{et}\quad
                   \mathrm{CS}(\mathcal{H}^{(i)}_{m},\ell_{m}^{(i)},s)=\mathrm{CS}(\F,\ell_{m}^{(i)},s)$.
\end{itemize}
\vspace{1mm}

\noindent Puisque $\mathrm{CS}(\F,\ell_{m}^{(1)},m)\mathrm{CS}(\F,\ell_{m}^{(2)},m)=1,$ nous avons $\mathrm{CS}(\mathcal{H}^{(1)}_{m},\ell_{m}^{(1)},m)\mathrm{CS}(\mathcal{H}^{(2)}_{m},\ell_{m}^{(2)},m)=1.$ Cette égalité et l'examen de la Table \ref{tab:CS(lambda)} donnent
\begin{align*}
\{\mathrm{CS}(\mathcal{H}^{(1)}_{m},\ell_{m}^{(1)},m),\,\mathrm{CS}(\mathcal{H}^{(2)}_{m},\ell_{m}^{(2)},m)\}=\{-2,-\tfrac{1}{2}\}\hspace{1mm};
\end{align*}
d'où $\mathrm{BB}(\F,m)=-\frac{1}{2}.$ Le point $m\in\Sigma^{2}$ étant arbitraire, $\Sigma^{2}$ est formé des $s\in\Sing\F$ tels que $\mathrm{BB}(\F,s)=-\frac{1}{2}$. Par suite le système (\ref{equa:Dar-BB-3}) se réécrit $\kappa_0+\kappa_1+\kappa_2=13$\, et \,$\frac{16}{3}\kappa_0+4\kappa_1-\frac{1}{2}\kappa_2=25$ dont l'unique solution est $(\kappa_0,\kappa_1,\kappa_2)=(0,7,6)$, c'est-à-dire que\hspace{1mm} $\Sing\F=\Sigma^{1}\cup\Sigma^{2},\hspace{2mm}\#\hspace{0.5mm}\Sigma^{1}=7$\hspace{2mm} et \hspace{2mm}$\#\hspace{0.5mm}\Sigma^{2}=6.$
\smallskip

\noindent Pour fixer les idées, nous supposons que $\mathrm{CS}(\mathcal{H}^{(1)}_{m},\ell_{m}^{(1)},m)=-2$ pour n'importe quel choix de $m\in\Sigma_2$; donc $\mathrm{CS}(\mathcal{H}^{(2)}_{m},\ell_{m}^{(2)},m)=-\frac{1}{2}.$ Dans ce cas, l'inspection de la Table \ref{tab:CS(lambda)} ainsi que les relations ($\mathfrak{a}$) et ($\mathfrak{b}$) conduisent à
\begin{align*}
&\#\hspace{0.5mm}(\Sigma^{1}\cap\ell_{m}^{(1)})=3,&&
\#\hspace{0.5mm}(\Sigma^{1}\cap\ell_{m}^{(2)})=2,&&
\Sigma^{2}\cap\ell_{m}^{(1)}=\{m\},&&
\Sigma^{2}\cap\ell_{m}^{(2)}=\{m,m'\}
\end{align*}
pour un certain point $m'\in\Sigma^{2}\setminus\{m\}$ vérifiant $\mathrm{CS}(\F,\ell_{m}^{(2)},m')=-\frac{1}{2}$. Ce point $m\hspace{0.1mm}'$ satisfait à son tour l'égalité $\Sigma^{2}\cap\ell_{m\hspace{0.1mm}'}^{(1)}=\{m\hspace{0.1mm}'\}$. Nous constatons que  $\ell_{m\hspace{0.1mm}'}^{(2)}=\ell_{m}^{(2)}$,\, $\ell_{m\hspace{0.1mm}'}^{(1)}\neq\ell_{m}^{(1)}$,\, $\ell_{m\hspace{0.1mm}'}^{(1)}\neq\ell_{m}^{(2)}$ et que ces trois droites distinctes satisfont $\Sigma^{2}\cap(\ell_{m}^{(1)}\cup\ell_{m}^{(2)}\cup\ell_{m\hspace{0.1mm}'}^{(1)})=\{m,m\hspace{0.1mm}'\}.$ Comme $\#\hspace{0.5mm}\Sigma^{2}=6=2\cdot3$, $\F$ possède $3\cdot3=9$ droites invariantes.
\smallskip

\noindent Posons $\Sigma^{1}\cap\ell_{m}^{(2)}=\{m_1,m_2\}$. Notons $\mathcal{D}_1,\mathcal{D}_2,\ldots,\mathcal{D}_6$ les six droites $\F$-invariantes qui restent; par construction chacune d'elles doit couper $\ell_{m}^{(1)}$ et $\ell_{m}^{(2)}$ en des points de $\Sigma^{1}.$ Par ailleurs, d'après la Remarque \ref{rem:droite-inva}, pour tout $s\in\Sing\F$ nous avons $\sigma(\F,s)\leq\tau(\F,s)+1\leq4.$ Donc par chacun des points $m_1$ et $m_2$ passent exactement trois droites de la famille $\{\mathcal{D}_1,\mathcal{D}_2,\ldots,\mathcal{D}_6\}$. Puisque $\#\hspace{0.5mm}(\Sigma^{1}\cap\ell_{m}^{(1)})=3$, $\Sigma^{1}\cap\ell_{m}^{(1)}$ contient au moins un point, noté $m_3$, par lequel passent précisément trois droites de la famille $\{\ell_{m\hspace{0.1mm}'}^{(1)},\mathcal{D}_1,\mathcal{D}_2,\ldots,\mathcal{D}_6\}$. Ainsi, pour $j=1,2,3$ nous avons $\sigma(\F,m_j)=4,$ ce qui implique que $\tau(\F,m_j)=3.$ L'hypothèse sur les singularités de $\F$ assure que $\nu(\F,m_j)=1$ pour $j=1,2,3.$ Il s'en suit que les singularités $m_1$, $m_2$ et $m_3$ sont radiales d'ordre $2$ de $\F$.
\smallskip

\noindent Par construction ces trois points ne sont pas alignés. Nous concluons en appliquant la Proposition~\ref{pro:Fermat-d}.
\end{proof}

\medskip

\noindent Dans \cite{MP13} les auteurs ont étudié les feuilletages de $\mathbf{F}(d)$ qui sont convexes à diviseur d'inflexion réduit; ils ont montré que l'ensemble formé de tels feuilletages est contenu dans $\mathbf{FP}(d)$, \emph{voir} \cite[Théorème 2]{MP13}.
Ces feuilletages sont à singularités non-dégénérées comme le montre l'énoncé suivant qui est une légère généralisation de \cite[Lemme~4.1]{MP13}.

\begin{lem}\label{lem:convexe-reduit-nd}
Tout feuilletage convexe sur $\pp$ à diviseur d'inflexion réduit est à singularités non-dégéné\-rées.
\end{lem}

\begin{proof}
Soit $\F$ un tel feuilletage et $s\in\Sing\F$ de multiplicité algébrique $\nu$. Fixons une carte affine $(x,y)$ telle que $s=(0,0)$; le germe $\F$ en $s$ est défini par un champ de vecteurs $\mathrm{X}$ du type $\mathrm{X}=\mathrm{X}_{\nu}+\mathrm{X}_{\nu+1}+\cdots$, où les $\mathrm{X}_{i}$ sont homogènes de degré $i$. Le diviseur d'inflexion $\mathrm{I}_{\F}$ de $\F$ est donné par l'équation
\[0=\left|\begin{array}{cc}\mathrm{X}(x) & \mathrm{X}(y)\\ \mathrm{X}^{2}(x) &\mathrm{X}^{2}(y)\end{array}\right|=P_{3\nu-1}(x,y)+\cdots,\]
où $P_{3\nu-1}(x,y)=\mathrm{X}_{\nu}(x)\mathrm{X}_{\nu}^{2}(y)-\mathrm{X}_{\nu}(y)\mathrm{X}_{\nu}^{2}(x)$ est un polynôme homogène (éventuellement nul) de degré $3\nu-1$.
Montrons d'abord que $\nu=1$. Les droites invariantes de $\F$ passant par l'origine sont contenues dans le cône tangent $yX_{\nu}(x)-xX_{\nu}(y)$ de $\mathrm{X}_{\nu}$ qui est un polynôme homogène de degré $\nu+1$. L'hypothèse sur $\F$ implique alors que $\nu=1$. Il s'en suit aussi  que le polynôme $P_{3\nu-1}$ n'est pas identiquement nul; par suite la partie linéaire $\mathrm{X}_{1}$ de $\mathrm{X}$ est saturée, ce qui
implique que la singularité $s$ est non-dégénérée.
\end{proof}
\smallskip

\noindent
À notre connaissance les seuls feuilletages convexes à diviseur inflexion réduit connus dans la littérature sont ceux qui sont présentés dans \cite[Table~1.1]{MP13}: le feuilletage $\F^{d}$ en tout degré et les trois feuilletages donnés par les $1$-formes
\[(2x^{3}-y^{3}-1)y\mathrm{d}x+(2y^{3}-x^{3}-1)x\mathrm{d}y\,,\]
\[(y^2-1)(y^2- (\sqrt{5}-2)^2)(y+\sqrt{5}x) \mathrm{d}x-(x^2-1)(x^2- (\sqrt{5}-2)^2)(x+\sqrt{5}y) \mathrm{d}y \, ,\]
\[(y^3-1)(y^3+7x^3+1) y \mathrm{d}x-(x^3-1)(x^3+7 y^3+1) x \mathrm{d}y\,,\]
qui sont de degré  $4$, $5$ et $7$ respectivement.
Dans \cite[Problème~9.1]{MP13} les auteurs demandent s'il y a d'autres feuilletages convexes à diviseur d'inflexion réduit. En combinant le Théorème~\ref{thm:Fermat} avec le Lemme~\ref{lem:convexe-reduit-nd} nous donnons une réponse négative en degré trois à ce problème.
\begin{cor}\label{cor3}
Tout feuilletage convexe de degré $3$ sur $\pp$ à diviseur d'inflexion réduit est linéairement conjugué au feuilletage de \textsc{Fermat} $\F^{3}$.
\end{cor}



\begin{thebibliography}{12}

\bibitem{BB72}
P.~Baum and R.~Bott.
\newblock Singularities of holomorphic foliations.
\newblock {\em J. Differential Geometry}, 7:279--342, 1972.

\bibitem{BD28}
W.~Blaschke and J.~Dubourdieu.
\newblock Invarianten von Kurvengeweben.
\newblock {\em Abh. Math. Sem. Univ. Hamburg}, 6:198--215, 1928.

\bibitem{BFM13}
A.~Beltr\'{a}n, M.~Falla~Luza, and D.~Mar\'{\i}n.
\newblock Flat 3-webs of degree one on the projective plane.
\newblock {\em Ann. Fac. Sci. Toulouse Math. (6)}, 23(4):779--796, 2014.

\bibitem{Bru00}
M.~Brunella.
\newblock {\em Birational geometry of foliations}.
\newblock First Latin American Congress of Mathematicians, IMPA, 2000.
\newblock Disponible sur
  \url{http://www.impa.br/opencms/pt/downloads/birational.pdf}.

\bibitem{CS82}
C.~Camacho and P.~Sad.
\newblock Invariant varieties through singularities of holomorphic vector fields.
\newblock {\em Ann. of Math. (2)}, 115(3):579--595, 1982.

\bibitem{CDGBM10}
D.~Cerveau, J.~D{\'e}serti, D.~Garba~Belko, and R.~Meziani.
\newblock G\'eom\'etrie classique de certains feuilletages de degr\'e deux.
\newblock {\em Bull. Braz. Math. Soc. (N.S.)}, 41(2):161--198, 2010.

\bibitem{Hen06}
A.~Hénaut.
\newblock Planar web geometry through abelian relations and singularities.
\newblock {\em Nankai Tracts Math.}, 11:269--295, 2006.

\bibitem{MP13}
D.~Mar\'{\i}n and J.~V. Pereira.
\newblock Rigid flat webs on the projective plane.
\newblock {\em Asian J. Math.} 17(1):163--191, 2013.

\bibitem{Mil99}
J.~Milnor.
\newblock {\em Dynamics in one complex variable: Introductory lectures},
\newblock Vieweg \& Sohn, Braunschweig, 1999.

\bibitem{Per01}
J.~V. Pereira.
\newblock Vector fields, invariant varieties and linear systems.
\newblock {\em Ann. Inst. Fourier (Grenoble)}, 51(5):1385--1405, 2001.

\bibitem{PP08}
J.~V. Pereira and L.~Pirio.
\newblock Classification of exceptional CDQL webs on compact complex surfaces.
\newblock {\em Int. Math. Res. Not. IMRN}, 12:2169--2282, 2010.

\bibitem{PP09}
J.~V. Pereira and L.~Pirio.
\newblock {\em An invitation to web geometry}.
\newblock IMPA, 2009.

\bibitem{Rip07}
O.~Ripoll.
\newblock Properties of the connection associated with planar webs and applications.
\newblock Prépublication \url{http://arxiv.org/abs/math/0702321}, 2007.

\end{thebibliography}
\end{document}